\documentclass[journal]{IEEEtran}
\usepackage{amsmath}    % need for subequations
\usepackage{graphics,epsfig,subfigure}    % need for figures
\usepackage{verbatim}   % useful for program listings
\usepackage{color}      % use if color is used in text
\usepackage{subfigure}  % use for side-by-side figures
\usepackage{amssymb}
\usepackage{wasysym}
\usepackage{epstopdf}
\usepackage{graphicx}
\usepackage{mathrsfs}

\newcommand{\prob}[1]{\text{Pr}\big\{#1\big\}}
\newcommand{\expect}[1]{\mathbb{E}\big\{#1\big\}}

\newcommand{\bv}[1]{{\boldsymbol{#1} }}
\newcommand{\script}[1]{{{\cal{#1} }}}

\newtheorem{lemma}{\textbf{Lemma}}

\newtheorem{theorem}{\textbf{Theorem}}

\newtheorem{assumption}{\textbf{Assumption}}
\newtheorem{definition}{\textbf{Definition}}

\allowdisplaybreaks
\begin{document}

%\title{Knowing a Bit Ahead in Matching Queueing Systems}
\title{The Value-of-Information in Matching with Queues}
\author{ Longbo Huang\\
longbohuang@tsinghua.edu.cn\\
IIIS, Tsinghua University
\thanks{This paper will be presented in part at the $16$th ACM International Symposium on Mobile Ad Hoc Networking and Computing (MobiHoc),  Hangzhou, China, June 2015.}
}

%\markboth{Draft}{Huang}

\maketitle

\begin{abstract} 
We consider the problem of \emph{optimal matching with queues} in dynamic systems and investigate the value-of-information. In such systems, the operators match tasks and resources  stored in queues, with the objective of maximizing the system utility of the matching reward profile, minus the average matching cost. This problem appears in many practical systems and the main challenges are the no-underflow constraints, and the lack of matching-reward information and system dynamics statistics. We develop two online matching algorithms: Learning-aided Reward optimAl Matching ($\mathtt{LRAM}$) and Dual-$\mathtt{LRAM}$ ($\mathtt{DRAM}$) to effectively resolve both challenges. Both algorithms are equipped with a learning module for estimating the matching-reward information, while $\mathtt{DRAM}$ incorporates an additional module for learning the system dynamics. We show that both algorithms achieve an $O(\epsilon+\delta_r)$ close-to-optimal utility performance for any $\epsilon>0$, while $\mathtt{DRAM}$ achieves a faster convergence speed and a better delay compared to $\mathtt{LRAM}$, i.e., $O(\delta_{z}/\epsilon + \log(1/\epsilon)^2))$ delay and $O(\delta_z/\epsilon)$ convergence under $\mathtt{DRAM}$ compared to $O(1/\epsilon)$ delay and convergence under $\mathtt{LRAM}$ ($\delta_r$ and $\delta_z$ are maximum estimation errors for reward and system dynamics). Our results reveal that information of different system components can play very different roles in algorithm performance and provide a systematic way for designing joint learning-control algorithms for dynamic systems. % with built-in 
%system learning module design
%may require different learning capabilities for achieving a desired goal. This provides useful insights into system learning module design.  
%different values of different system information to algorithm performance. ==
%We show that by incorporating different learning modules, both $\mathtt{LRAM}$ achieves a utility that is 

%different queues and the system operator decides a matching every time. Each matching generates a random reward with unknown statistics a prior,  and the system objective, different from traditional flow utility maximization problems, is to optimize a system utility that on the achieved reward profile. 

%We first introduce the notion of a $(T, \delta, P_{\delta})$-\emph{learning module}, which is an abstract representation of statistical learning algorithms and features a learning time $T$, a maximum error $\delta$ and a probability of guarantee $P_{\delta}$. We then 

%We design two online algorithms $\mathtt{LRAM}$ and $\mathtt{DRAM}$. $\mathtt{LRAM}$ utilizes the statistical information 

% that utilize learning modules for learning the statistics of the random reward as well as the system dynamics. We show that $\mathtt{LRAM}$ achieves an 

%At every time, the system operator decides the matching between them and the matching results in a random reward. 

% which is a general framework that  

%and investigate the benefits of having prediction of future dynamics. The matching queueing system 
\end{abstract}

\section{Introduction}\label{section:intro}
Matching is a fundamental problem that  appears in resource allocation in various systems across  different areas. For instance,  network switch scheduling \cite{mckeown96}, online advertising \cite{matchingnowbook}, crowdsourcing \cite{lbc-2010}, ride sharing \cite{uber}, cloud computing \cite{cloud-mwm12},   and inventory control \cite{neelyhuang_assembly}. Hence,  efficient  matching algorithms are of great importance to system control. 

%====. 

In this paper, we study the problem of optimal matching with queues in a dynamic environment  with unknown matching reward statistics. Specifically, we consider a system consists of a set of \emph{task queues} and a set of \emph{resource queues}, which store different types of workload and different types of resources that come into the system according to some random processes. At every time, the system operator decides how to match the resources to the pending workload. Each matching incurs a  \emph{cost}  that depends on the resource allocated and random factors in the system, e.g., changing channel conditions in a downlink system,  time-varying prices in inventory control, or fluctuating payment requirements in crowdsourcing. On the other hand, the matching also generates a \emph{reward}, which is  random with an unknown distribution  determined by the amount of tasks resolved and the system condition. The objective is to design a matching strategy that carefully manages the resources and tasks, so as to achieve optimal system utility, which is a function of the achieved reward profile, subject to the constraint  that all  tasks are fulfilled timely. 
 
This is a general problem and models the aforementioned application scenarios. However, it is  very challenging to solve. First, the system utility is a function of the matching reward, which means that it is affected by when and how much resource is actually matched to the tasks and is only indirectly related to the traffic rates. This differs significantly from traditional flow utility optimization problems \cite{neelysuperfast}, \cite{jiang-csma10}, and requires both careful admission control to avoid instability and appropriate matching to achieve good utility. Second, since each matching action is  rewarded based on the actual amount of tasks   resolved, the matching scheme must ensure that there are nonzero tasks and nonzero resources in the queues, i.e., \emph{no-underflow}. This constraint is complex and is mostly tackled with dynamic programming, which can have high computational complexity. Third, the system is dynamic and the statistics of system conditions and  reward functions are unknown beforehand. This requires that the matching scheme can efficiently learn the sufficient statistics of the randomness  and adapt to the changing environment. 

In addition to resolving the above challenges, we also  take one step further and try to investigate the \emph{value-of-information} in such matching systems with queues, by explicitly considering the impact of information on algorithm performance. 
Existing works on stochastic system control either focus on systems with perfect a-prior  information, e.g., \cite{tan-downlink-09}, \cite{xiaojun-order-delay11}, or rely on stochastic approximation techniques that do not require such information, e.g., \cite{hou-delay-10}, \cite{neelynowbook}. While the proposed solutions are effective, they do not capture how information affects  algorithm design and performance, and do not provide  interfaces for integrating the fast-developing ``data science'' tools, e.g.,  data collecting methods and machine learning algorithms, \cite{jordan-data-report13}, \cite{pattern-recognition-bishop}, into system control. 

To provide a rigorous quantification of the value of information, we first introduce an abstract notion of a \emph{learning module}, which represents a general information learning algorithm and features a  learning accuracy level $\delta$ (maximum error), a learning time $T_{\delta}$, and the probability of learning accuracy guarantee $P_{\delta}$. We then design two online matching algorithms: Learning-aided Reward optimAl Matching ($\mathtt{LRAM}$) and Dual-$\mathtt{LRAM}$ ($\mathtt{DRAM}$). $\mathtt{LRAM}$ utilizes a single $(T_{\delta_r}, \delta_r, P_{\delta_r})$ learning module for estimating the reward statistics and achieves an $O(\epsilon+\delta_r)$ system utility, for any $\epsilon>0$, while ensuring an $O(1/\epsilon)$ delay bound and an $O(1/\epsilon)$ algorithm convergence time,  defined to be the time taken for the algorithm to enter the optimal control state. $\mathtt{DRAM}$ incorporates an additional $(T_{\delta_z}, \delta_z, P_{\delta_z})$ learning module for the random system state distribution and  guarantees a similar $O(\epsilon+\delta_r)$ system utility. Moreover, $\mathtt{DRAM}$ is able to achieve an $O(\delta_{z}/\epsilon + \log(1/\epsilon)^2)$ delay bound and an $O(\delta_z/\epsilon)$ algorithm convergence time, which can be significantly faster compared to $\mathtt{LRAM}$. 

%=briefly discuss why delay and convergence are both very important=
%We further propose a sampling-based learning module that guarantees $\delta_z=\epsilon^$
%We also show that $\delta_z=\epsilon^c$ can be achieved by a sampl

Our results reveal an interesting fact that the reward  information largely determines the utility performance, while the system dynamics  information  greatly affects delay and  algorithm convergence. This indicates that \emph{information of different system components can have different impacts on algorithm  performance, and may require different learning power for achieving a desired goal}. 
Closest to our paper is the recent work \cite{huang-learning-sig-14}, which considers joint learning and control. Our framework allows much more general learning methods and resolves the no-underflow constraints. We also quantify the values of different system information. %Our results also 

We summarize the main contributions as follows: 
\begin{enumerate}
 
\item We propose a matching queueing system model, which can model general  resource-task matching problems in stochastic systems. To explicitly quantify the value of information in such systems, we  introduce an abstract notion of a $(T_{\delta}, \delta, P_{\delta})$-\emph{learning module} that captures key characteristics of general  learning algorithms and provides interfaces for bringing the information learning aspect into system control.  
 
\item We design two learning-aided matching algorithms $\mathtt{LRAM}$ and $\mathtt{DRAM}$. We show that with a single $(T_{\delta_r}, \delta_r, P_{\delta_r})$-module for learning reward statistics, $\mathtt{LRAM}$ achieves an $O(\epsilon+\delta_r)$ utility, while ensuring an $O(1/\epsilon)$ delay bound and an $O(1/\epsilon)$ algorithm convergence time. $\mathtt{DRAM}$ adopts an additional $(T_{\delta_z}, \delta_z, P_{\delta_z})$-module for learning the system state distribution, and guarantees a similar $O(\epsilon+\delta_r)$ system utility, while achieving an $O(\delta_{z}/\epsilon + \log(1/\epsilon)^2)$ delay  and an $O(\delta_z/\epsilon)$ convergence time. We also construct two $(O(1/\epsilon^c), O(\epsilon^{c/2}), 1-O(\epsilon^{\log(1/\epsilon)}))$ online learning modules based on sampling ($c>0$). Combining them with $\mathtt{DRAM}$, one achieves a fast $O(1/\epsilon^{1-c/2}+1/\epsilon^c)$ convergence time with $c<1$ (existing algorithms require $\Theta(1/\epsilon)$).

\item Our algorithm design approach provides a low-complexity way to tackle multiple simultaneous no-underflow constraints in systems and jointly optimize utilities that are not defined on flow rates.  The development of $\mathtt{DRAM}$ also demonstrates how general  learning algorithms can be combined with queue-based control (stochastic approximation) to achieve superior delay performance and accelerate algorithm convergence speed. 

%\item We construct a general $(O(1/\epsilon^c), \epsilon^{c/2}, 1-O(\epsilon^{\log(1/\epsilon)}))$ online learning module based on sampling. Applying === 
\end{enumerate}

%It is also very challenging ===. What is more, we also want to understand the fundamental impact of information,  which can be obtained via different approaches such as online learning or data mining, on system performance. 

%Our results reveal the following facts: === 

%ridesharing, assembly line, cloud computing etc

The rest of the paper is organized as follows. We first list a few motivating examples in Section \ref{section:example}. We then present the matching system model in Section \ref{section:model}. The algorithm design approach and the two algorithms $\mathtt{LRAM}$ and $\mathtt{DRAM}$ are presented in Section \ref{section:algorithm}. Analysis is carried out in Section \ref{section:analysis} and simulation results are presented in Section \ref{section:simulation}. We then conclude the paper in Section \ref{section:conclusion}. %To facilitate reading, all proofs are presented in the appendices. 

\section{Motivating Examples}\label{section:example}
\textbf{Crowdsourcing:} In a crowdsourcing application, e.g., crowdsourcing query search \cite{lbc-2010} or ride-sharing \cite{uber}, tasks of different types (task) arrive at the server and are assigned to workers (resource). %, depending on pre-specified criteria and worker availability. 
The workers then carry out the tasks. Depending on the workers' qualifications, the types of jobs, and the instantaneous system condition (state), e.g., whether a query requestor is in a hurry due to weather, the requestors receives certain reward, e.g., satisfaction, and the workers receive payments. The objective of the system  is to design a matching scheme, so as to maximize the system utility, which is a function of the achieved requestor reward profile. 

\textbf{Energy Harvesting Networks:} In an energy harvesting network, e.g., \cite{tap-energy-ton14}, \cite{chen-energy-ton14}, nodes are responsible for  transmitting data (task) and can harvest energy (resource) from the environment. At every time, each node decides how much energy to allocate for transmission and determines traffic scheduling. Depending on the time-varying channel condition (state), the amount of energy enables certain processing results. The objective is to design a joint energy management and scheduling algorithm, so as to maximize traffic utility and ensure that no energy outage happens. 

\textbf{Online Advertisement:} In an online advertising system, \cite{matchingnowbook},  \cite{online-ad-opt-2011},  advertisers deposit  money (task) into their accounts at the advertising platform. Queries (resource) for different keywords arrive in the system and the server decides which advertiser's ads to show, based on their relevance  to the keywords and the available budget of the advertisers.  
Depending on the chosen ad  and the user's condition (state), e.g., location or mood, a business transaction may take place. 
%Upon receiving the ads, the user making the query collects certain information and 
%the users collect certain useful information. Their satisfactions depend on the types of the queries and their instantaneous condition, which in turn affect the revenue of the advertising platform. 
The goal of the system is to design an ad matching scheme, so as to maximize the system's utility, which is a function of the average income profile from advertisers.  

%\textbf{Cloud Computing:} 

\textbf{Cloud Computing:} In a cloud computing platform, e.g., \cite{cloud-mwm12},  computing resources (resource), e.g., CPU, memory, are assigned to virtual machine instances (task) for processing arriving job requests. The quality of experience of a requestor depends on the job completion quality, which is affected by system conditions such as background task level (state) and the user status. 
The objective here is to design a resource allocation policy, such that the overall quality of service is maximized. 

In all these examples, the underlying problem is indeed matching with queues. Below, we present the general model. 
%In an inventory control system, raw materials (task and resource) are combined to produce different products. The resources are 
%Depending on the types of resource and products, their availability and selling prices are different. Producing a certain good may require 
%Thus, the operator needs at each time what product to produce and to sell, as to to maximize the revenue \cite{=neely=}. 

%\textbf{Downlink Transmission:} In a downlink system, the base station (BS) supports a set of flows by transmitting their packets (task) to the mobile users. Since the BS has always-on power supply (resource), its main objective is to decide power allocation under different channel conditions, so as to maximize the flow utility it can achieve. 

\section{System Model}\label{section:model}
We consider a discrete-time system shown in Fig. \ref{fig:system}. In this system, there are two sets of queues,  \emph{task queues} and  \emph{resource queues}, and a central server (called operator below), which coordinates resource allocation and scheduling in the system. Time is divided into unit-size slots, i.e., $t\in\{0, 1, ...\}$. 
\begin{figure}[cht]
%\vspace{-.15in}
\begin{center}
\includegraphics[height=1.5in, width=3.2in]{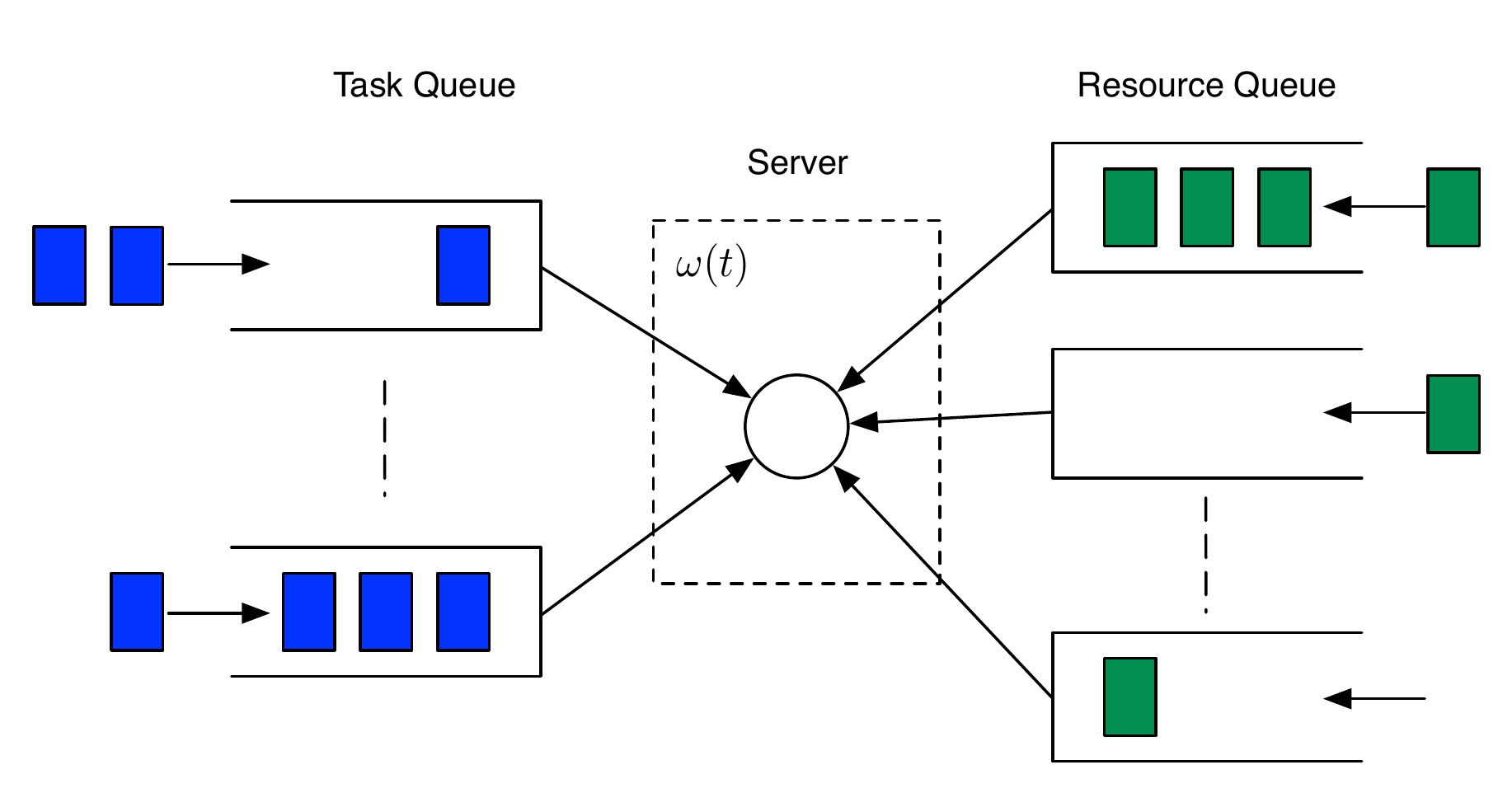}
\vspace{-.1in}
\end{center}
\vspace{-.1in}
\caption{The matching queueing system. }\label{fig:system}
\vspace{-.2in}
\end{figure}

\subsection{Tasks and Resources}
The task queues store jobs that come into the system and are waiting to be served by the server. We assume there are $N$ types of tasks and denote the set of task queues by $\script{Q}=\{\script{Q}_1, ..., \script{Q}_N\}$. We use $A_n(t)$ to denote the amount of new tasks arrivaling at $\script{Q}_n$ at time $t$ and assume that $0\leq A_n(t)\leq A_{\max}$.  
%$0\leq A_n(t)\leq A_{\max}$. 
% 
We then define the arrival vector  $\bv{A}(t) = (A_1(t), ..., A_N(t))$. 
%  and assume that it is i.i.d. every time slot, but the components can be correlated.  We also use $\lambda_n=\expect{A_n(t)}$ to denote the expected rate of the arrivals to $\script{Q}_n$. 
% 
In many systems, arrivals to the system may not always all be admitted due to congestion control, e.g., when all servers are busy. We model this by using $0\leq R_n(t)\leq A_n(t)$ to denote the actual admitted traffic to $\script{Q}_n$ at time $t$. %We then use $\script{R}_{\bv{A}(t)}$ to denote set of feasible admission actions 
We then use $Q_n(t)$ to denote the amount of tasks stored at $\script{Q}_n$ at time $t$ and denote $\bv{Q}(t)=(Q_1(t), ..., Q_N(t))$ the task queue vector. 

%\subsection{Resources}
The resource queues, on the other hand,  hold the resources the system collects over time. There are $M$ types of system resources and we denote the resource queues by $\script{H}=\{\script{H}_1, ..., \script{H}_M\}$. We similarly let $e_m(t)$ be the amount of new resource arriving at $\script{H}_m$ with $0\leq e_m(t)\leq h_{\max}$. We also use $H_m(t)$ to denote the amount of resource $m$ the system current holds and denote $\bv{H}(t)=(H_1(t), ..., H_M(t))$ the resource queue vector. 

%, and assume that $\bv{e}(t)=(e_1(t), ..., e_M(t))$ is i.i.d. with $\expect{e_m(t)}=\eta_m$. %, and $h_m(t)\in\{0, 1\}$. 
%$0\leq h_m(t)\leq h_{\max}$. 
%We also use $H_m(t)$ to denote amount of resources $m$ the system current holds. 

In many systems, it is feasible (and sometimes necessary) to control the amount of resources in the system, e.g.,  to avoid too many workers waiting in crowdsourcing. We model this decision by using $h_{m}(t) \in[0,  e_m(t)]$ to denote the actual amount of type $m$ resource admitted.  
%We also assume that the decision vector $\bv{h}(t)$ is chosen from an available action set $\script{I}_{\bv{e}(t)}$. By specifying different action sets $\script{I}$, it is possible to model many different system settings. For example, we can set $\script{I}_{\bv{e}(t)}=\{\bv{e}(t)\}$ to model the case when the resource queues always admit all incoming resource. 
For now, it is also convenient to temporarily assume that the queues are all of unlimited sizes. We will later show that our algorithms ensure that finite buffer sizes are sufficient.

%At every time $t$, we use $Q_n(t)$ to denote the amount of unserved jobs in task queue $\script{Q}_n$, and uses $H_m(t)$ to denote the amount of type $m$ resources that the system currently accumulates. 
%We also denote $\bv{Q}(t)=(Q_1(t), ..., Q_N(t))$ and $\bv{H}(t)=(H_1(t), ..., H_M(t))$ the system queue vectors. 

\subsection{System State and Resource allocation}
We assume that the system has a time-varying condition, e.g., the channel conditions in a downlink system, or the expected happiness measures of human users in a crowdsourcing system. We call this condition the \emph{system state} and model it by a random state variable $\omega(t)$. Note that $\omega(t)$ represents the \emph{aggregate} system condition. %We assume that $\omega(t)$ is i.i.d. and takes values in a finite set $\Omega=\{\omega_1, \omega_2, ..., \omega_K\}$, and it is observable to the system operator. We also use $\pi_k$ to denote the probability that $\omega(t)=\omega_k$. 

Denote $\bv{z}(t)=(\bv{A}(t), \bv{e}(t), \omega(t))$. In this paper, we assume that $\bv{z}(t)$ is i.i.d. and takes values in $\script{Z}=\{\bv{z}_1, ..., \bv{z}_K\}$. We then denote $\pi_k=\prob{\bv{z}(t)=\bv{z}_k}$.  Note that this allows arbitrary dependency among $\bv{A}(t)$, $\bv{e}(t)$, and $\omega(t)$. %In this case, we also conveniently express the corresponding admission sets as $\script{R}_{k}$ and $\script{I}_k$. 

At every time $t$,   the system operator determines the amount of resource to allocate to serving each queue. We denote this decision by a \emph{matching matrix} $\bv{b}(t)=(b_{mn}(t), m, n)$, 
where $b_{mn}(t)$ denotes the type $m$ resource allocated to queue $n$. When $\bv{z}(t)=\bv{z}_k$, $\bv{b}(t)$  takes values from a finite discrete set $\script{B}_{k}\subset\mathbb{R}^N_+$.\footnote{This assumption is made to simplify the learning algorithm description (Section \ref{section:imperfect}). Our results can likely be extended to the case when $\{\script{B}_{k}, k\}$ are general compact sets in $\mathbb{R}^N_+$. }% can be approximated by a discrete set.}
We define $b_{\max}\triangleq \max_{\bv{b}\in\script{B}_k, k}\|\bv{b}\|_{\infty}$ the maximum amount of resource allocated to any queue at any time. %Without loss of generality, we assume that $\script{B}_{\bv{z}(t)}$ is a finite set. One example will be $\script{B}_{\bv{z}(t)}=\{\bv{b}:0\leq b_{mn}\leq b_{\max}\}$. 
It is clear that at any time $t$, we must have: 
\begin{eqnarray}
\sum_n b_{mn}(t) \leq H_m(t), \quad\forall\, m. \label{eq:no-underflow}
\end{eqnarray}
This is because one cannot spent more resource than what is available. In the following, we call (\ref{eq:no-underflow}) the \emph{no-underflow} constraint. 
Depending on the system state and the  resource allocation decision, each task queue gets a service rate $\mu_n(t)\triangleq \mu_n(\bv{z}(t), \bv{b}(t))$. %, where $\bv{b}_n(t)=(b_{mn}(t), m)$ is the resource vector  to task queue $n$. 
We assume $\mu_n(\bv{z}(t), \bv{b}(t))\in[0, \mu_{\max}]$ for all $\bv{z}(t)$ and $\bv{b}(t)$ and that $\{\mu_n(\bv{z}(t), \bv{b}(t))\}_{n\in\script{N}}$ are known to the operator. Also, they satisfy that $\mu_n(\bv{z}(t), \bv{0})=0$  for all $\bv{z}(t)$, and if $\mu_n(\bv{z}(t), \bv{b})>0$, then 
\begin{eqnarray}
\mu_n(\bv{z}(t), \bv{b})\geq \beta^l_{\mu}\min_{m:b_{mn}>0} b_{mn},\label{eq:mu-prop1}
\end{eqnarray}
for some $ \beta^l_{\mu}>0$. 
Moreover, if two vectors $\bv{b}$ and $\bv{b}'$ are such that $\bv{b}'$ is obtained by setting $b_{mn}$ in $\bv{b}$ to zero, then, 
\begin{eqnarray}
\hspace{-.3in}\mu_n(\bv{z}, \bv{b}) \leq \mu_n(\bv{z}, \bv{b}') +  \beta^u_{\mu} b_{mn},\,\forall\, n. \label{eq:mu-prop2}
\end{eqnarray}
Note that (\ref{eq:mu-prop1}) and (\ref{eq:mu-prop2}) are not  restrictive. They simply require that nonzero resource is needed for getting a positive service rate, and that a positive rate is upper and lower bounded by linear functions of the resources allocated.  
%In the special case when each resource queue only supplies one task queue, we $$

%Without loss of generality, we assume that $A_n(t)$, $R_n(t)$, and $\mu_n(t)$ all take integer values. 

%, which we assume to be an i.i.d. random variable . 

\vspace{-.1in}
\subsection{Matching Cost and Reward}
In every time slot, due to resource expenditure, there is a \emph{matching cost} associated with the resource allocation decision. %For instance, in online advertising, the advertiser pays for the ad being displayed, or in a crowdsourcing platform, the task requestor pays the platform for helping to accomplish a job. 
We model this by  denoting  $c(t)=c(\bv{z}(t), \bv{b}(t))$ the cost for choosing the resource vector $\bv{b}(t)$. This cost can represent, e.g., cost for purchasing raw materials in inventory control, or payments to workers in a crowdsourcing application.  
One example is $c(\bv{z}(t), \bv{b}(t))=\sum_{nm}c_m(\bv{z}(t))b_{nm}(t)$, where $c_m(\bv{z}(t))$ denotes the per-unit resource price for type $m$ resource under state $\bv{z}(t)$. 
We assume that $c(\bv{z}(t), \bv{b}(t))\in[0, c_{\max}]$ for all time and it is known to the system operator. Also, if $\bv{0}\preceq\bv{b}_1\preceq\bv{b}_2$ (entrywise-less), when $c(\bv{z}(t), \bv{b}_1)\leq c(\bv{z}(t), \bv{b}_2)$. 

Every time a matching is completed, the operator collects a \emph{matching reward}, e.g., a customer conversion due to an ad, or user satisfaction due to job completion. 
% 
%To model the fact that in many systems, such reward can be due to both task completion and resource assignment, we  ==== 
% 
We model this by denoting the reward collected at time $t$ from type $n$ tasks by   $\kappa_n(t)$. We assume that $\kappa_n(t)\in[0, r_{\max}]$ is an i.i.d. random variable given $\bv{z}(t)$ and $\bv{b}(t)$, and its mean is determined by the reward function $r_n(\bv{z}(t), \tilde{\mu}_n(t))\triangleq\expect{\kappa_n(t)\left.|\right. \bv{z}(t), \bv{b}(t), \bv{Q}(t)}$, where $\tilde{\mu}_n(t) = \min[Q_n(t), \mu_n(t)]$ denotes the \emph{actual} amount of tasks completed. We assume that $r_n(\bv{z}(t), \tilde{\mu}_n(t))$ satisfies: 
\begin{eqnarray}
r_n(\bv{z}(t), \mu)\leq r_n(\bv{z}(t), \mu'), \,\text{if}\,\, \mu\leq\mu', 
\end{eqnarray}
and denote $\bv{r}=\{ r_n(\bv{z}, \mu_n(\bv{z}, \bv{b}_z)), \bv{z}\in\script{Z}, \bv{b}_z\in\script{B}_{\bv{z}} \}$ the reward matrix. Since each $\script{B}_{\bv{z}}$ is finite, $\bv{r}$ is also finite.  

%In many case, the system also gets reward when matching 
%= 

Different from existing works, e.g., \cite{online-ad-opt-2011}, \cite{huangneelypricing-ton}, we do not assume any prior knowledge of the functions $r(\bv{z}(t), \mu)$.\footnote{This is different from the $\bv{\mu}$ functions, which measure how much resources are spent and can  typically  be observed by the system controller.}
This is quite  common in practice. For example, in crowdsourcing applications, it is often unknown a-prior how qualified a worker is for a certain type of tasks;  or in  online advertising, one often does not know the conversion probabilities beforehand. %we do not often know the click-through-rate beforehand and may only be able to observe its realized values every time. 

\subsection{Queueing}
%At every time $t$, we use $Q_n(t)$ to denote the amount of unserved jobs in task queue $\script{Q}_n$. We then similarly uses $H_m(t)$ to denote the amount of type $m$ resource that the system currently accumulates. We  denote $\bv{Q}(t)=(Q_1(t), ..., Q_N(t))$ and $\bv{H}(t)=(H_1(t), ..., H_M(t))$ the system queue vectors. 
From the above, we see that the queue vectors $\bv{Q}(t)$ and $\bv{H}(t)$ evolve according to: 
\begin{eqnarray}
\hspace{-.2in}Q_n(t+1) &=& \max[Q_n(t) -\mu_n(t), 0] +R_n(t),\,\forall\,n, \label{eq:q-dyn} \\
\hspace{-.2in}H_m(t+1) &=& H_m(t)-\sum_{n}b_{mn}(t)+ h_m(t), \,\forall\, m. \label{eq:h-dyn}
\end{eqnarray}
Notice that there is no $\max[\cdot, \cdot]$ operator in (\ref{eq:h-dyn}). This is due to the no-underflow constraint (\ref{eq:no-underflow}).  
In our paper, we say that a queue vector process $\bv{x}(t)\in\mathbb{R}^d_+$ is stable if it satisfies: \footnote{In this paper, we assume that all limits exist with probability $1$. Our results can be extended to more general cases with $\liminf$ or $\limsup$ arguments. }
\begin{eqnarray}
x_{\text{av}} \triangleq \lim_{t\rightarrow\infty}\frac{1}{t}\sum_{\tau=0}^{t}\sum_{n=1}^d\expect{x_n(\tau)}<\infty. 
\end{eqnarray}

\subsection{Utility Optimization}
The system's utility is determined by a function of the average matching reward profile. Specifically,  define $\overline{r}_{n}\triangleq \lim_{t\rightarrow\infty}\frac{1}{t}\sum_{\tau=0}^{t-1}\expect{r_n(\tau)}$. The system utility is given by:  
\begin{eqnarray}
U_{\text{total}}(\overline{\bv{r}})\triangleq \sum_nU_n(\overline{r}_{n}). % + \sum_mU_m(\overline{y}_m). \label{eq:reward}
\end{eqnarray}
Here each $U_n(r)$ is an increasing concave function with $U_n(0)=0$ and $U'(0)<\infty$. We denote $\beta\triangleq\max_{n, r}(U_n(r))'$ the maximum first derivative of the utility functions.  
We also define the following system cost due to resource expenditure:  
\begin{eqnarray}
C_{\text{total}}\triangleq \lim_{t\rightarrow\infty}\frac{1}{t}\sum_{\tau=0}^{t-1}\expect{c(\tau)}.\label{eq:cost}
\end{eqnarray}
%The system objective is to design 
% 
We say that a matching algorithm $\Pi$  is feasible if for all time $t$, it selects  $0\preceq\bv{R}(t)\preceq\bv{A}(t)$, $0\preceq\bv{h}(t)\preceq\bv{e}(t)$, and $\bv{b}(t)\in\script{B}_{\bv{z}(t)}$, and it ensures constraint (\ref{eq:no-underflow}) for all time. 
Our objective  is to design a feasible policy $\Pi$, so as to: 
\begin{eqnarray}
\max: && f_{\text{av}}\triangleq U_{\text{total}}(\overline{\bv{r}}) - C_{\text{total}}  \label{eq:opt-control-obj}\\
\text{s.t.} &&  Q_{\text{av}}<\infty,\, H_{\text{av}}<\infty. 
\end{eqnarray}
We denote the optimal solution value as $f_{\text{av}}^*$.  
Here  the queue stability constraints are to ensure that the tasks and resources do not stay in the queue forever. This is important in many cases. For instance, in an energy harvesting network, it is important to ensure timely packet delivery, or in a crowdsourcing system, it is desirable to keep the worker waiting time short. 
%assign workers jobs once they enter the resource queue awaiting for jobs. 

%maximize $U_{\text{total}} - c_{\text{total}}$, subject to (\ref{eq:no-underflow}) and queue stability of $Q_n(t)$ and $H_m(t)$. 

\subsection{Discussion on the Model}
Due to the general matching reward function, our problem is different from a flow utility maximization problem, e.g., \cite{neelysuperfast}, which is a special case when $r_n(t)=\tilde{\mu}_n(t)$. 
By tuning the parameters of the model, our model can model all the examples  in  Section \ref{section:example}. For example, by choosing $U_{\text{total}}=\sum_n\alpha_n\overline{r}_n$, the system models the revenue maximziation problem in online advertisement systems. By choosing $\mu_n(t) = 1_{[b_1(t)>b_1^{\min}]}1_{[b_2(t)>b_2^{\min}]}$, our model can represent a cloud computing system, where a computing task requires two types of resources. 

Problem (\ref{eq:opt-control-obj})  is very challenging. First of all, the no-underflow constraint  (\ref{eq:no-underflow}) requires a very careful selection of control actions, because actions in a slot can affect action feasibility in later slots. 
Problems of this kind are often tackled with dynamic programming, whose computational complexity can be extremely high when the action space is large. 
Secondly, the reward function $r_n(\bv{z}(t), \tilde{\mu}_n(t))$ is unknown and is dependent on $\bv{Q}_n(t)$. This makes the problem very different from existing utility maximization works, e.g., \cite{neelysuperfast}, \cite{jiang-csma10}. 
Thirdly, due to the more and more stringent user requirements on service quality, it is more desirable to ensure  small queueing delay.  
% to be linear functions of $\bv{r}$, the system reward becomes the revenue. Allowing a general concave function, on the other hand, makes it possible to take certain task ``fairness'' into account. 

%One simple example of the system can 
%For instance, the jobs can be data packets that must be processed, and the 

%\section{Optimality without Buffering}
%We start our investigation with a 

\section{Optimal Matching}\label{section:algorithm}
In this section, we present our matching algorithms. We will first present an ideal algorithm, which assumes full knowledge of  the reward functions $r_n(\bv{z}(t), \mu)$ and will serve as a basic building block. Even in this case, we will see that the problem is highly nontrivial due to the existence of the no-underflow constraint (\ref{eq:no-underflow}) and the dependency of $r_n(t)$ on $\bv{Q}(t)$. 

\subsection{With Full Reward Information} 
To start, we first introduce an auxiliary variable $\gamma_n(t)\in[0, r_{\max}]$ and  create for each $n$ a  \emph{deficit queue} $d_n(t)$ that evolves as follows: 
\begin{eqnarray}
d_n(t+1) = \max[d_n(t) - \kappa_n(t), 0] + \gamma_n(t),  \label{eq:d-dyn}
\end{eqnarray}
with $\bv{d}(0)=\bv{0}$. 
Note that the input into $d_n(t)$ is $\kappa_n(t)$ instead of $r_n(t)$. The deficit queue $d_n(t)$ measures how much the actual reward profile is currently lagging behind the target value (due to randomness). 
%This is so because the system utility is defined as a function of the average reward profile $\bv{r}_{\text{av}}$. 

Then, we  denote $f(t) \triangleq \sum_nU_n(\gamma_n(t))-c(t)$ the instantaneous system utility minus cost and denote  $\bv{y}(t)\triangleq(\bv{Q}(t), \bv{H}(t), \bv{d}(t))$.  
We also define a Lyapunov function as follows: 
\begin{eqnarray}
L(t) = \frac{1}{2}\|\bv{Q}(t)-\bv{\theta}_1\|^2 + \frac{1}{2}\|\bv{H}(t)-\bv{\theta}_2\|^2+\frac{1}{2}\|\bv{d}(t)\|^2,
\end{eqnarray}
where $\|\cdot\|$ is the euclidean norm and $\bv{\theta}_1=\theta_1\cdot\bv{1}^N$ and $\bv{\theta}_2=\theta_2\cdot\bv{1}^M$ with $\bv{1}^k\in\mathbb{R}^k$ being the vector with all components being $1$, and  $\theta_1$ and $\theta_2$ are constants that will be specified later. 
% 
%\[L(t)=\frac{1}{2}\sum_n(Q_n(t)-\theta_1)^2+d_n^2(t)) + \frac{1}{2}\sum_m(H_m(t)-\theta_2)^2,\] 
We define the one-slot utility-based conditional Lyapunov drift $\Delta_V(t) \triangleq \expect{L(t+1)-L(t) - Vf(t)\left.|\right. \bv{y}(t)}$. Using the queueing dynamics (\ref{eq:q-dyn}), (\ref{eq:h-dyn}),  and (\ref{eq:d-dyn}), we obtain the following lemma, in which $V\geq1$ is a tunable parameter introduced for controlling the tradeoff between system utility and service delay (explained later). 
\begin{lemma}\label{lemma:drift}
Under any feasible policy, we have:
\begin{eqnarray}
\hspace{-.3in}&&\Delta_V(t)\leq G - V\sum_n\expect{U_n(\gamma_n(t)) -d_n(t)\gamma_n(t)\left.|\right. \bv{y}(t) } \label{eq:drift} \\ 
\hspace{-.3in}&&\qquad\qquad\quad + \sum_n (Q_n(t)-\theta_1)\expect{R_n(t)\left.|\right. \bv{y}(t)}  \nonumber\\
\hspace{-.3in}&&\qquad\qquad\quad+ \sum_m(H_m(t)-\theta_2)\expect{h_m(t)\left.|\right. \bv{y}(t)}  \nonumber\\
\hspace{-.3in}&&  \qquad\qquad \quad  +\,\expect{Vc(t) -  \sum_m(H_m(t)-\theta_2)\sum_{n}b_{mn}(t)  \nonumber \\
\hspace{-.3in}&&\qquad\qquad  -\sum_n(Q_n(t)-\theta_1) \mu_n(t)-  \sum_{n} d_n(t)r_n(t)\left.|\right. \bv{y}(t)}.  \nonumber
\end{eqnarray}
Here $G\triangleq N(A_{\max}^2 +\mu_{\max}^2+2r_{\max}^2)+Mh_{\max}^2+MN^2b_{\max}^2$ does not depend on $V$, and the expectations are taken over the randomness in the system as well as in the policy. $\Diamond$
\end{lemma}
\begin{proof}
%Omitted due to space limit. Please see \cite{huang-matching-14}.  
See Appendix A. 
\end{proof}

We now construct our algorithm by minimizing the right-hand-side (RHS) of the drift (\ref{eq:drift}). %Do so, we obtain the following algorithm. 

\underline{\textbf{\textsf{Reward optimAl Matching ($\mathtt{RAM}$):}}} At every time $t$, observe $\bv{z}(t)$ and  $\bv{y}(t)$. Do: 
\begin{enumerate}
%\vspace{-.05in}
\item \textbf{Quota:} For each $n$, choose $\gamma_n(t)$ by solving: 
\begin{eqnarray}
\hspace{-.3in}\max: \,\, VU_n(\gamma_n(t)) -d_n(t)\gamma_n(t), \,\,\text{s.t.}\,\, 0\leq\gamma_n(t)\leq r_{\max}. \label{eq:quota}
\end{eqnarray}
%\vspace{-.05in}
\item \textbf{Admission:} For each $n$, if $Q_n(t)<\theta_1$, let $R_n(t)=A_n(t)$; otherwise $R_n(t)=0$. Similarly, for each $m$,  if $H_m(t)<\theta_2$, let $h_m(t)=e_m(t)$; otherwise $h_m(t)=0$.
%Choose $\bv{R}(t)$ and $\bv{h}(t)$ according to: 
%\begin{eqnarray}
%\bv{R}(t)&=&\max_{\bv{R}\in\script{R}_{\bv{A}(t)}}\sum_n( Q_n(t) - \theta_1 )R_n(t) \\ 
%\bv{R}(t)&=&\max_{\bv{R}\in\script{R}_{\bv{A}(t)}}\sum_n( Q_n(t) - \theta_1 )R_n(t).  
%\end{eqnarray}
%
 
\item \textbf{Resource:} Choose the resource allocation vector $\bv{b}(t)$ by solving: 
\begin{eqnarray}
\hspace{-.5in}\min: && \Psi_{r}(\bv{b})\triangleq Vc(\bv{z}, \bv{b}) -  \sum_m(H_m(t)-\theta_2)\sum_{n}b_{mn} \label{eq:Psi}\\
\hspace{-.5in}&& \qquad \quad\,\,\,\, -\sum_n(Q_n(t)-\theta_1) \mu_n(\bv{z}(t), \bv{b}) \nonumber\\
\hspace{-.5in}&&\qquad \quad \,\,\,\, -  \sum_{n} d_n(t)r_n(\bv{z}(t), \mu_n(\bv{z}(t), \bv{b}))\nonumber\\
\hspace{-.5in}\text{s.t.} && \bv{b}\in\script{B}_{\bv{z}(t)}, \text{Constraint } (\ref{eq:no-underflow})
\end{eqnarray}

\item \textbf{Queueing:} Update $\bv{Q}(t)$, $\bv{H}(t)$, and $\bv{d}(t)$,  according to (\ref{eq:q-dyn}), (\ref{eq:h-dyn}), and (\ref{eq:d-dyn}), respectively.  $\Diamond$
\end{enumerate}

%=talk about challenges here=
%In the special case shen $\mu_n(\bv{z}(t), \bv{b}) = \mu_n(\bv{z}(t), \bv{b}_n)$, we see that the above simplifies to ===.  
Note that in (\ref{eq:Psi}) we have used $r_n(\bv{z}(t), \mu_n(t))$ instead of $r_n(\bv{z}(t), \tilde{\mu}_n(t))$. We will see in our later analysis that our algorithm automatically guarantees $\tilde{\mu}_n(t) = \mu_n(t)$. This is very useful, for otherwise the algorithm performance will be very hard to analyze. We also emphasize here that the introduction of $\theta_1$ and $\theta_2$ are important. It can be seen in the admission step here that if  $\theta_1=\theta_2=0$, i.e., without $\theta_1$ and $\theta_2$,  no task or resource will be admitted at the first place and the algorithm will not even proceed!

\subsection{With Reward Information Learning}\label{section:imperfect}
Here we consider the case when one does not have full reward information and  provide an algorithm that can integrate general learning methods for estimating $\bv{r}$. %\note that this does not mean that we learn the entire $r_n(\bv{z}, \mu)$ function due to the possibility that $Q_n(t)<\mu$).  

To also investigate the impact of learning on algorithm design and performance, we first define learning capability. Specifically, for any general matrix $\bv{W}$ and a  learning algorithm $\Gamma$  that outputs an estimation $\hat{\bv{W}}$,  we denote its maximum estimation error by: 
\begin{eqnarray}
\delta_{w}\triangleq \|\hat{\bv{W}} -  \bv{W} \|_{\max}, 
\end{eqnarray}
where $\|\bv{x}\|_{\max}\triangleq \max_{ij}|x_{ij}|$. Then, the formal definition of a \emph{learning module} is as follows.
\begin{definition}
An algorithm $\Gamma$ is called a $(T_{\delta}, P_{\delta}, \delta)$-learning module, if (i) it completes learning in $T_{\delta}$ time,  (ii) it guarantees  $\prob{\delta_{w}<\delta}\geq P_{\delta}$, and (iii) for any $T\geq0$, $P_{\delta}$ does not decrease if the algorithm is run for $T_{\delta}+T$ time. $\Diamond$
\end{definition}

Here $T_{\delta}$ can be both random or deterministic depending on the termination rules. This definition is general and captures key features of learning algorithms. 
With this definition, having perfect knowledge at the beginning can be viewed as having an $(0, 1, 0)$-learning module. 

We now present an optimal matching algorithm for general systems that do not possess perfect knowledge of  $\bv{r}$ and need to rely on some learning algorithms for estimation.  In the algorithm, we use $\beta_{\hat{r}}$ to denote the maximum ``derivative'' of  the estimated $\hat{\bv{r}}$ with respect to any $b_{mn}$. Specifically,  we assume that  if $\bv{b}$ and $\bv{b}'$ are such that $\bv{b}'$ is obtained by setting one $b_{mn}$ in $\bv{b}$ to zero. Then, 
\begin{eqnarray}
%\hspace{-.3in}\mu_n(\bv{z}, \bv{b}_n) &\leq& \mu_n(\bv{z}, \bv{b}'_n) +  \beta_{\mu} b_{mn},\,\forall\, n \\
\hspace{-.3in}\hat{r}_n(\bv{z}, \mu_n(\bv{z}, \bv{b})) &\leq& \hat{r}_n(\bv{z}, \mu_n(\bv{z}, \bv{b}')) + \beta_{\hat{r}} b_{mn}, \,\forall\, n. 
\end{eqnarray}
Since both $\script{Z}$ and $\{\script{B}_{\bv{z}}\}$ are finite, we see that  $\beta_{\hat{r}}$ exists and is $\Theta(1)$ (possibly depends on $\hat{\bv{r}}$). %\footnote{Note that it is possible and common to continuously update the estimation during the matching phase. It can be shown that }

\underline{\textbf{\textsf{Learning-aided Reward OptimAl Matching ($\mathtt{LRAM}$):}}} 
\begin{enumerate}
\item (\textbf{Learning}) Apply any $(T_{\delta_r}, P_{\delta_r}, \delta_r)$-learning module $\Gamma_r$. Terminate at $t=T_{\delta_r}$ and output $\hat{\bv{r}}$. 
 
\item (\textbf{Matching}) Set $\bv{Q}(T_{\delta_r}+1)=0$, $\bv{H}(T_{\delta_r}+1)=0$, and $\bv{d}(T_{\delta_r}+1)=0$. Choose $\theta_1$ and $\theta_2$ according to: 
\begin{eqnarray}
\theta_1 &=&(h_{\max}+ (V\beta+r_{\max})\beta_{\hat{r}})/\beta^l_{\mu} +\mu_{\max} \label{eq:theta-value1}\\
\theta_2 &=&(V\beta+r_{\max})\beta_{\hat{r}} + r_{\max}\beta^u_{\mu} +Nb_{\max}. \label{eq:theta-value2}
\end{eqnarray}
Run $\mathtt{RAM}$ with $\hat{\bv{r}}$. $\Diamond$ 
\end{enumerate}
In $\mathtt{LRAM}$, we explicitly separate the algorithm into two disjoint phases. This is chosen to facilitate presentation and analysis. Doing so also does not change the order of the overall algorithm convergence time and performance. % 
We can also transform $\mathtt{LRAM}$ and $\mathtt{DRAM}$ below into continuous-learning versions, e.g.,  \cite{huang-learning-sig-14}, and  update estimations from time to time, e.g., using sliding-window estimation or frame-based estimation. %Similar performance results can be proven. 
It is also worth noting that $\theta_1$ and $\theta_2$ can be computed beforehand easily. This is a  feature useful for implementation.

%When learning can naturally be separated from control, e.g., \cite{huang-learning-sig-14}, $\mathtt{LRAM}$ can also be transformed into one that performs continuous learning during the matching phase. %  ., e.g., \cite{huang-learning-sig14}, when learning can naturally be separated from control. 

\subsection{With System State Information Learning}\label{section:imperfect2}
In the previous section, we describe how the estimated reward information $\hat{\bv{r}}$ can be naturally integrated into a matching algorithm.  
Here we consider the case when a learning module is also applied to learning the statistics of the system state $\bv{z}(t)$. 
%While in the previous case, incorporating learned information is quite straightforward, here 
Our result here generalizes the dual-learning approach proposed in \cite{huang-learning-sig-14} to handle underflow and to allow general learning methods. 

To start, we  define the following optimization problem:  
\begin{eqnarray}
\hspace{-.3in}\max: && \Phi\triangleq V[\sum_nU_n(\gamma_n) - Cost] \label{eq:obj}\\
\text{s.t.} && \gamma_n\leq r_n\triangleq \sum_k\pi_kr_n(\bv{z}_k, \mu_n(z_k, \bv{b}^{k}))\label{eq:rn}\\
\hspace{-.3in}&& Cost\triangleq \sum_k\pi_kc(\bv{z}_k, \bv{b}^{k})\label{eq:cost}\\
\hspace{-.3in}&& \sum_k\pi_kR^{k}_n =  \sum_k\pi_k\mu_n(\bv{z}_k, \bv{b}^{k}),\,\forall\,n\label{eq:rate}\\
\hspace{-.3in}&& \sum_k\pi_k\sum_{n}b_{mn}^{k} =  \sum_k\pi_kh^{k}_m,\,\forall\,m \label{eq:resource}\\
\hspace{-.3in}&& 0\leq \gamma_n\leq r_{\max}, \bv{b}^{k}\in\script{B}_k \\
\hspace{-.3in}&& 0\preceq\bv{R}^k\preceq\bv{A}^k, 0\preceq\bv{h}^k\preceq\bv{e}^k. \label{eq:set}
\end{eqnarray}

Problem (\ref{eq:obj}) can intuitively be viewed as a way to solve our matching problem. The equalities in (\ref{eq:rate}) and (\ref{eq:resource}) are due to fact that only tasks that are actually served generate reward and the no-underflow constraint (\ref{eq:no-underflow}). However, a scheme obtained by solving (\ref{eq:obj}) may not be implementable due to  (i) it ignores the no-underflow constraint, and (ii) it assumes that all resources allocated are fully utilized, i.e., using $\mu_n(z_k, \bv{b}^{k})$ in (\ref{eq:rn}). We will also see later that, it requires a much larger learning time for such a scheme to have the right statistics for achieving a performance comparable to ours. %, the learning time is much larger. %the convergence time is much longer. 
% 
%Resolving such problems is very hard and usually requires dynamic programming.  

We now obtain the  dual problem for (\ref{eq:obj}) as follows. 
\begin{eqnarray}
\min: \,\,\, g(\bv{\alpha}^{d}, \bv{\alpha}^{q}, \bv{\alpha}^{h})\,\,\text{s.t.}\,\, \bv{\alpha}^{d}\succeq\bv{0}, \bv{\alpha}^{q}\in\mathbb{R}^N, \bv{\alpha}^{h}\in\mathbb{R}^M, \label{eq:dual-fun}
\end{eqnarray}
where $g(\bv{\alpha}^{d}, \bv{\alpha}^{q}, \bv{\alpha}^{h})=\sum_k\pi_kg_k(\bv{\alpha}^{d}, \bv{\alpha}^{q}, \bv{\alpha}^{h})$ is the dual function, and  $g_k(\bv{\alpha}^{d}, \bv{\alpha}^{q}, \bv{\alpha}^{h})$ is defined as: 
\begin{eqnarray}
\hspace{-.3in}&&g_k(\bv{\alpha}^{d}, \bv{\alpha}^{q}, \bv{\alpha}^{h}) = \sup_{\bv{R}, \bv{\gamma}, \bv{b}, \bv{h}}\bigg\{V[\sum_nU_n(\gamma_n) - Cost] \label{eq:dual-zk} \\
\hspace{-.3in}&&\qquad\qquad\qquad - \sum_n\alpha^d_n[r_n(\bv{z}_k, \mu_n(z_k, \bv{b})) - \gamma_n]  \nonumber\\
\hspace{-.3in}&&\qquad- \sum_n\alpha^q_n[ R_n -  \mu_n(\bv{z}_k, \bv{b})]  - \sum_m\alpha^h_m[ \sum_{n}b_{mn} -  h_n] \bigg\}. \nonumber
\end{eqnarray}
%Here the $\sup$ is taken over all feasible $\bv{R}$, $\bv{\gamma}$ and $\bv{b}$. 
Note that $g_k(\bv{\alpha}^{d}, \bv{\alpha}^{q}, \bv{\alpha}^{h})$ is indeed the dual function for state $\bv{z}_k$. 
With (\ref{eq:dual-fun}), we now present our algorithm, which  integrates system state information learning into control. 

\underline{\textbf{\textsf{Dual learning-aided Reward optimAl Matching ($\mathtt{DRAM}$):}}} 
\begin{enumerate}
 
\item (\textbf{Learning}) Apply any $(T_{\delta_{r}}, P_{\delta_{r}}, \delta_{r})$-learning module $\Gamma_{r}$ for $\bv{r}$ and any $(T_{\delta_{z}}, P_{\delta_{z}}, \delta_{z})$-learning module $\Gamma_{z}$ for $\bv{z}$. Terminate at $T_{L}=\max(T_{\delta_r}, T_{\delta_z})$ and output $\hat{\bv{r}}$ and $\hat{\bv{\pi}}$. Choose $\theta_1$ and $\theta_2$ according to: 
\begin{eqnarray}
\theta_1 &=& (h_{\max}+ (V\beta+r_{\max})\beta_{\hat{r}})/\beta^l_{\mu} +\mu_{\max} \\
\theta_2 &=& (V\beta+r_{\max})\beta_{\hat{r}} + r_{\max}\beta^u_{\mu} +Nb_{\max}. 
\end{eqnarray}
%\vspace{-.05in} 
\item (\textbf{Dual learning}) Solve the \emph{empirical dual problem}: 
\begin{eqnarray}
\hspace{-.3in}\min: && \sum_k\hat{\pi}_k\hat{g}_k(\bv{\alpha}^{d}, \bv{\alpha}^{q}-\bv{\theta}_1, \bv{\alpha}^{h}-\bv{\theta}_2) \label{eq:dual-learning}\\
\hspace{-.3in}\text{s.t.} && \bv{\alpha}^{d}\succeq\bv{0}, \bv{\alpha}^{q}\in\mathbb{R}^N, \bv{\alpha}^{h}\in\mathbb{R}^M. \nonumber 
\end{eqnarray}
Here $\hat{g}_k(\bv{\alpha}^{d}, \bv{\alpha}^{q}-\bv{\theta}_1, \bv{\alpha}^{h}-\bv{\theta}_2)$ is defined in (\ref{eq:dual-zk}) with $\hat{\bv{r}}$. 
Denote the optimal solution as $\hat{\bv{\alpha}}^*=(\hat{\bv{\alpha}}^{d*}, \hat{\bv{\alpha}}^{q*}, \hat{\bv{\alpha}}^{h*})$. 
%\vspace{-.05in}
\item (\textbf{Matching}) Set  $\bv{Q}(T_{L}+1)=0$, $\bv{H}(T_{L}+1)=0$, and $\bv{d}(T_{L}+1)=0$. For all $t\geq T_L+1$, define: 
\begin{eqnarray}
\hat{\bv{Q}}(t)&=&\bv{Q}(t)+\hat{\bv{\alpha}}^{q*} - \bv{\zeta}_N\\
\hat{\bv{H}}(t)&=&\bv{H}(t)  + \hat{\bv{\alpha}}^{h*} -  \bv{\zeta}_M   \\
\hat{\bv{d}}(t)&=&\bv{d}(t)+\hat{\bv{\alpha}}^{d*} - \bv{\zeta}_N. 
\end{eqnarray}
where $ \bv{\zeta}_k$ is a vector in $\mathbb{R}^k$ with all elements being  $\zeta \triangleq 2\max(\delta_{z} V\log(V)^2, \log(V)^2)$. Run $\mathtt{RAM}$ with $\hat{\bv{r}}$, $\hat{\bv{Q}}(t)$, $\hat{\bv{H}}(t)$, and $\hat{\bv{d}}(t)$. If the resulting $\bv{b}(t)$ from (\ref{eq:Psi}) violates (\ref{eq:no-underflow})  for some $m$, change $\{b_{mn}(t)\}_{n\in\script{N}}$ to $\{\tilde{b}_{mn}(t)\}_{n\in\script{N}}$ with $\sum_n\tilde{b}_{mn}(t)=h_m(t)$ and drop $\mu_n(\bv{z}(t), \bv{b}(t))$ tasks from each $n$ that has $b_{mn}(t)>0$. $\Diamond$ \footnote{In actual implementation, one can still serve the tasks with the actual allocated resource. }
%modify the chosen $\bv{b}(t)$ into any $\bv{b}'$ that ==modify the proof!==.  
% 
\end{enumerate}

$\mathtt{DRAM}$  first utilizes  learning to  obtain an empirical distribution, which is usually  crude but fast  at the beginning. Then, it transforms to queue-based control (or more generally, stochastic approximation), by obtaining an empirical optimal multiplier via dual learning. It then starts from the empirical multiplier and rely on queue-based control to learn the true optimal operation point. The procedure is shown in Fig. \ref{fig:explain}. This combination avoids the slow convergence regime of statistical learning and the slow start of stochastic approximation, and algorithms so developed can achieve much faster convergence and superior delay.   
\begin{figure}[cht]
\centering
%\vspace{-.1in}
\includegraphics[height=1.5in, width=2.8in]{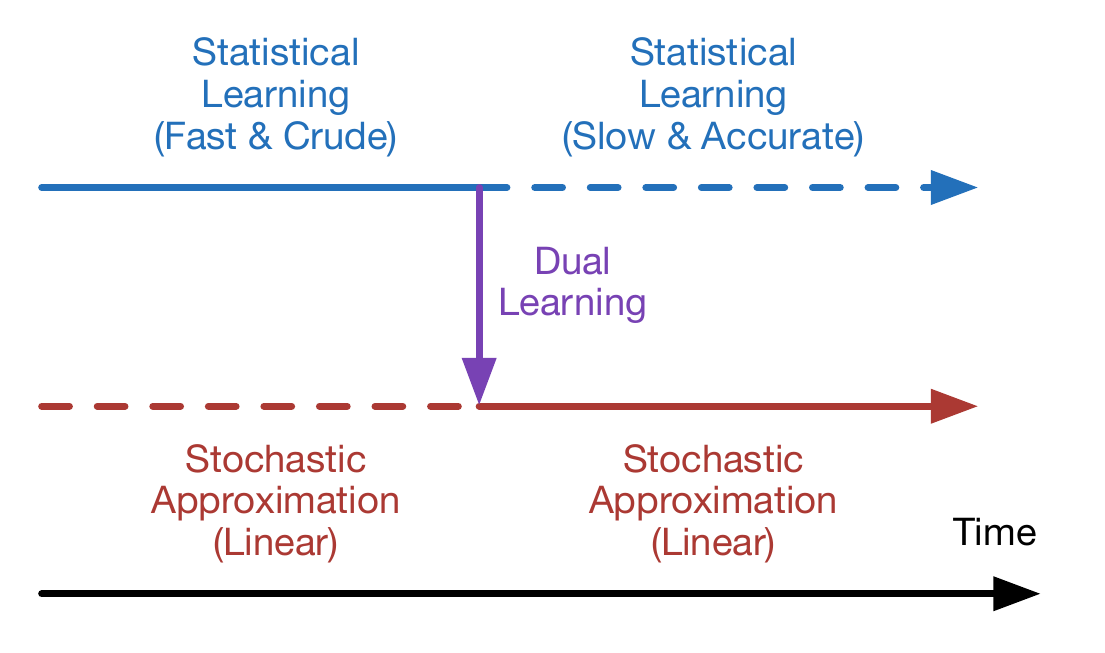}
\vspace{-.15in}
\caption{$\mathtt{DRAM}$ combines the fast regime of statistical learning and the smoothness of  queue-based control (more generally, stochastic approximation).  }\label{fig:explain}
\vspace{-.15in}
\end{figure}

%we develo
%=talk about joint statistical learning+stochastic approx=
%Note that the dual learning step is critical here. It has been shown in \cite{huangneely_dr_tac} that the queue vector $(\hat{\bv{Q}}(t), \hat{\bv{H}}(t), \hat{\bv{d}}(t))$ are attracted to the optimal Lagrange multiplier of (\ref{eq:dual-fun}). However, if we use pure queue-based approach for tracking them, e.g., $\mathtt{RAM}$, it takes a long time for the algorithm to converge, because $\bv{Q}(t)$, $\bv{H}(t)$, and $\bv{d}(t)$ changes by $\Theta(1)$ every time. Thus, the dual learning step here is to first obtain an empirical optimal multiplier quickly, hoping that it is close to the optimal Lagrange multipliers. Then, we start from the empirical multiplier and rely on queue-based control to learn the true optimal operation point. 

Also note here that the dropping step is introduced to ensure zero underflow, by giving up the service rates and reward at that particular slot. We will see in later analysis and simulation that such an event almost never happens and hence does not affect performance. 

%=discuss using a sliding window, discuss it may be desirable to use finite window=
%and use that to accelerate convergence. 

%It is tempting to think that one can solve the matching problem by 

%It is also important to note here that it is 
%Intuitively, it tries to learn the approximate optimal Lagrange multiplier, who actual attracts the vectors %$\hat{\bv{Q}}(t)$, $\hat{\bv{H}}(t)$, and $\hat{\bv{d}}(t)$. Thus, by replacing the actual queue sizes with the 

%=add a figure here? =

%=explain why a fast convergence is important=

%=remark that one cannot solve the opt problem to achieve the performance=

%\textbf{Remark:} We first remark that 

\subsection{Sampling-based Learning Module}
Here we describe two sampling-based learning modules for estimating $\bv{r}$ and $\bv{\pi}$. We first describe a threshold-based sampling module for estimating $\bv{r}$. In the module, we use $s(\bv{z}, \bv{b}, t)$ to denote the number of times the pair $(\bv{z}, \bv{b})$ is sampled, i.e., adopt $\bv{b}$ when $\bv{z}(t)=\bv{z}$,  up to time $t$. We also denote $s_{\min}(t)=\min_{(\bv{z}, \bv{b})} s(\bv{z}, \bv{b}, t)$. 

\underline{\textsf{Threshold-Based Sampling $\mathtt{TBS}(s_{th})$}}: Every time $t$,  sample the resource allocation vector $\bv{b}^* \in \arg\min_{\bv{b}} s(\bv{z}(t), \bv{b}, t)$ until  $s_{\min}(t)\geq s_{th}$. If terminate at $T_{tbs}$, output $\hat{r}_n(\bv{z}, \bv{b})=\sum_{t=1}^{T_{tbs}}1_{[\bv{z}(t)=\bv{z}, \bv{b}(t)=\bv{b}]} \kappa_n(t)/s(\bv{z}, \bv{b}, T_{tbs})$. 
$\Diamond$

Here $1_{[x]}$  is the indication function of $x$. 
The learning module $\mathtt{TBS}(s_{th})$ is very intuitive. It tries to balance the sampling frequencies of all $(\bv{z}, \bv{b})$ until every pair is sampled at least $s_{th}$ times. In this case,  the learning algorithm running time $T_{tbs}$ is random. In the following, we look at a deterministic time sampler for estimating $\bv{\pi}$. This module sets a sampling time threshold $T_{th}$. 

\underline{\textsf{Time-Limited Sampling $\mathtt{TLS}(T_{th})$}}: Observe $\bv{z}(t)$ for $T_{th}$ slots. Output $\hat{\pi}_k=\sum_{t=1}^{T_{th}}1_{[\bv{z}(t)=\bv{z}_k]}/T_{th}$ for all $k$. 
$\Diamond$

The following lemma shows the performance of the two modules. 
\begin{lemma}\label{lemma:module}
The two learning modules satisfy: 
\begin{enumerate}
\item[(a)]  $\mathtt{TBS}(s_{th})$: $\expect{T_{tbs}} = \Theta(s_{th})$ and with probability $1-O( e^{   - \frac{ s_{th} \log(s_{th})^2  }{  2(s_{th}r_{\max}^2 + r_{\max}\sqrt{s_{th}}/3)} } )$, $\delta_r =  \frac{\log(s_{th})}{\sqrt{s_{th}}}$. 
\item[(b)] $\mathtt{TLS}(T_{th})$: With probability $1- O(e^{-\frac{T^{th} \log(T_{th})^2}{ 2(T_{th}+\sqrt{T_{th}}/3) }})$, $\delta_z= \frac{\log(T_{th})}{\sqrt{T_{th}}}$. $\Diamond$ 
\end{enumerate}
\end{lemma}
\begin{proof}
%Omitted due to space limit. Please see \cite{huang-matching-14}.
See Appendix B. 
\end{proof}
By choosing $s_{th}=T_{th}=V^c$, we can guarantee $\delta_r=\delta_z=c\log(V)V^{-c/2}$ with  $P_{\delta_r}$ and $P_{\delta_z}$ being $1-O(V^{-\log(V)})$. 

%from Part (a) that by choosing $s_{th}=\log(V)^2$, we can guarantee that $\delta_r=\log\log(V)/\log(V)$ and $P_{\delta_r}=1-O(1/V)$. Similarly, if we choose $T_{th}=V^c$, we can guarantees $\delta_z=V^{-c/2}$ with probability $1-O()$. 

%present a sampling-based learning module for learning a general matrix $\bv{W}$, assuming that every time one can only choose one $(i, j)$ and obtain a single i.i.d. value $w_{ij}$ with $0\leq w_{ij}(t)\leq w_{\max}$ and ===

\section{Performance Analysis}\label{section:analysis}

In this section, we present the performance results of $\mathtt{LRAM}$ and $\mathtt{DRAM}$. We start with some definitions and assumptions. For notation simplicity, we  denote $\bv{\alpha}=(\bv{\alpha}^{d}, \bv{\alpha}^{q}, \bv{\alpha}^{h})$ and write $\bv{\alpha}-\bv{\theta}=(\bv{\alpha}^{d}, \bv{\alpha}^{q}-\bv{\theta}_1, \bv{\alpha}^{h}-\bv{\theta}_2)$. Also, to indicate the different distributions and reward functions used, we use $\hat{g}(\bv{\alpha})$ to denote the dual function when $\bv{r}$ is replaced by $\hat{\bv{r}}$ in (\ref{eq:dual-zk}) and the distribution is given by $\bv{\pi}$. Then, we use $g^{\hat{\pi}}(\bv{\alpha})$ and $\hat{g}^{\hat{\pi}}(\bv{\alpha})$ to denote the dual function with distribution $\hat{\bv{\pi}}$ and $\bv{r}$, and with $\hat{\bv{\pi}}$ and $\hat{\bv{r}}$, respectively.

%prove carry out our analysis  under a polyhedral system structure, which often appears in systems with discrete action sets. The formal definition of the polyhedral structure is as follows \cite{huangneely_dr_tac}.
%We first consider the case when the system has a \emph{polyhedral} structure defined as follows.
%\vspace{-.06in}

\subsection{Preliminaries}
We define the following polyhedral system structure: 
\begin{definition}
A system is polyhedral with parameter $\rho>0$ if the dual function $g(\bv{\alpha})$ satisfies:
\begin{eqnarray}
g(\bv{\alpha}^*) \leq g(\bv{\alpha})-\rho\|\bv{\alpha}^* -  \bv{\alpha}\|. \label{eq:polyhedral}
\end{eqnarray}
Here $\bv{\alpha}^*$ is an optimal solution of (\ref{eq:dual-fun}). $\Diamond$ 
\end{definition}
The polyhedral structure typically appears in practical systems, especially when the system action sets are discrete (see  \cite{huangneely_dr_tac} for more discussions). Note that (\ref{eq:polyhedral}) holds for all $V$ if it holds under any $V$, in particular $V=1$.

In our analysis, we make the following assumptions. 
%\vspace{-.1in}
\begin{assumption}\label{assumption:bdd-LM}
There exist constants $\epsilon_{r}, \epsilon_z=\Theta(1)>0$ such that for any valid state distribution $\hat{\bv{\pi}}$ and reward statistics $\hat{\bv{r}}$ with $\|\hat{\bv{\pi}} - \bv{\pi} \|\leq \epsilon_z$ and $|| \bv{r} - \hat{\bv{r}} ||\leq \epsilon_r$, there exist a set of actions $\{\bv{\gamma}^{k}\}_{k=1,..., |\script{Z}|}$, $\{\bv{R}^{k}_i\}_{k=1,..., |\script{Z}|}^{i=1,2, ..., \infty}$,  $\{\bv{b}^{k}_i\}_{k=1,..., |\script{Z}|}^{i=1,2, ..., \infty}$, and  $\{\bv{h}^{k}_i\}_{k=1,..., |\script{Z}|}^{i=1,2, ..., \infty}$,  and variables $\lambda^{k}_{i}\geq0$ with $\sum_i\lambda^{k}_{i}=1$ for all $k$  (possibly depending on $\hat{\bv{\pi}}$ and $\hat{\bv{r}}$), such that:
\begin{eqnarray}
\gamma_n -  \sum_k\hat{\pi}_k\sum_{i}\lambda^k_i\hat{r}_n(\bv{z}_k, \mu_n(z_k, \bv{b}_i^{k}))\leq - \eta_0, \label{eq:rate-service-a}
%\mathbb{E}_{\bv{\pi'}}\big\{  \gamma_m^{\bv{\eta}} -  \sum_{k} \lambda^{k}_{\bv{\eta}}\sum_{n, j} I^{\bv{\eta}}_{nmj}   \hat{\beta}_{nm}(A^{\bv{\eta}}_{nm}, j, S^{\bv{\eta}}_{n}))\big\}\leq -\vartheta_0, \forall\, m, \label{eq:slackness1}
\end{eqnarray}
where $\eta_0=\Theta(1)>0$ is independent of $\hat{\bv{\pi}}$ and $\hat{\bv{r}}$, and that 
\begin{eqnarray}
\hspace{-.3in}&&  \sum_k\hat{\pi}_k\sum_i\lambda^k_i R^{k}_{in} =  \sum_k\hat{\pi}_k\sum_i\lambda^k_i\mu_n(\bv{z}_k, \bv{b}^{k}_i) \label{eq:rate-a}\\
\hspace{-.3in}&& \sum_k\hat{\pi}_k\sum_i\lambda^k_i \sum_{n}b_{imn}^{k} =  \sum_k\hat{\pi}_k\sum_i\lambda^k_i h^{k}_{im} \label{eq:resource-a}
 %$\mathbb{E}_{\bv{\pi'}}\{\sum_{k} \lambda^{k}_{\bv{\eta}}\sum_{n, m, j} I^{\bv{\eta}, k}_{nmj} (1- A^{\bv{\eta}}_{nm})\} \leq c_n-\vartheta_0, \forall\, n$. %. , \label{eq:slackness2}
\end{eqnarray}
where $0<\sum_k\hat{\pi}_k\sum_i\lambda^k_i R^{k}_{in}<\mathbb{E}_{\hat{\bv{\pi}}}\{A_n(t)\}$ as well as $0< \sum_k\hat{\pi}_k\sum_i\lambda^k_i h^{k}_{im} <\mathbb{E}_{\hat{\bv{\pi}}}\{e_m(t)\}$. $\Diamond$
%where $\vartheta_0=\Theta(1)>0$ is independent of $\bv{\pi}'$ and $\hat{\bv{\beta}}$. $\Diamond$
\end{assumption} 
%\vspace{-.1in}
\begin{assumption}\label{assumption:share-poly}
For any $\hat{\bv{\pi}}$ and $\hat{\bv{r}}$ with $\|\hat{\bv{\pi}} - \bv{\pi} \|\leq \epsilon_z$ and $||  \hat{\bv{r}} - \bv{r}  ||\leq \epsilon_{r}$, if $g(\bv{\alpha})$ is polyhedral with parameter $\rho$, then $\hat{g}^{\hat{\pi}}(\bv{\alpha})$  is also polyhedral with parameter $\rho$. %Moreover, $\hat{g}_{\hat{\bv{\pi}}}(\bv{h}, \bv{q})$ has an unique optimal solution. 
$\Diamond$
\end{assumption}
%\vspace{-.1in}
 \begin{assumption}\label{assumption:unique}
For any $\hat{\bv{\pi}}$ and $\hat{\bv{r}}$ with $\|\hat{\bv{\pi}} - \bv{\pi} \|\leq \epsilon_z$ and $||  \hat{\bv{r}} - \bv{r}  ||\leq \epsilon_{r}$,   $\hat{g}^{\hat{\pi}}(\bv{\alpha})$ has a unique optimal solution over $\mathbb{R}^{M+2N}$. 
$\Diamond$
\end{assumption}

In the network optimization literature, e.g., \cite{neelynowbook}, \cite{javad-flow-stabile}, Assumption \ref{assumption:bdd-LM} is commonly assumed with  $\epsilon_r=\epsilon_z=0$.   By allowing $\epsilon_r, \epsilon_z>0$, we assume that systems with similar statistics have similar stability regions. Assumption \ref{assumption:share-poly} assumes that systems with similar statistics share a similar dual structure. This is not restrictive. In fact, when action sets are discrete, it is often the case that $g_k(\bv{\alpha})$ is polyhedral, which usually leads to a polyhedral structure of $\hat{g}^{\hat{\pi}}(\bv{\alpha})$. The uniqueness assumption is also often guaranteed by the utility maximization structure, e.g., \cite{eryilmaz_qbsc_ton07}.

%We first have the following queueing bounds. 
%
%\begin{lemma} \label{lemma:q-bdd}
%Under $\mathtt{LRAM}$, we have for all $t\geq T_{\delta}+1$ that: 
%\begin{eqnarray}
%d_n(t)&\leq& V\beta +r_{\max}, \,\forall\, n\\
%Q_n(t)&\leq& \theta_1 +r_{\max}, \,\forall\, n\\
%H_m(t)&\leq& \theta_2 +h_{\max}, \,\forall\, m. \quad\Diamond
%\end{eqnarray}
%\end{lemma}
%\begin{proof} 
%See Appendix B. 
%\end{proof}
\subsection{Queue and Utility Performance}
We first summarize the performance of $\mathtt{LRAM}$. %  in the following theorem.  
%The following theorem provides the utility performance. 
%=discuss the challenge=
\begin{theorem}\label{theorem:lram}
Suppose $T_{\delta_r}<\infty$ with probability $1$. Under $\mathtt{LRAM}$, we have for all $t\geq T_{\delta_r}+1$ that: 
\begin{eqnarray}
d_n(t)&\leq& d_{\max}\triangleq V\beta +r_{\max}, \,\forall\, n\\
Q_n(t)&\leq& Q_{\max}\triangleq\theta_1 +r_{\max}, \,\forall\, n\\
H_m(t)&\leq& H_{\max}\triangleq\theta_2 +h_{\max}, \,\forall\, m.  
\end{eqnarray}
Moreover, we have with probability $P_{\delta_r}$ that, 
\begin{eqnarray}
f_{\text{av}}^{\mathtt{LRAM}} \geq f_{\text{av}}^* - \frac{G+ r_{\max} \delta_r}{V} - 2N\beta\delta_r. \,\Diamond\label{eq:utility-per}
\end{eqnarray}
\end{theorem}
\begin{proof}
See Appendix C. 
\end{proof}

The last term in (\ref{eq:utility-per}) involves the estimation error $\delta_r$. This can be viewed as the performance loss  due to inaccurate reward information. %== It is interesting to notice that the above bounds does not depend on $\hat{\bv{r}}$. 
We remark here that the deterministic queueing bounds are important for both algorithm implementation and performance guarantee. This is so because errors in reward function  estimation will be amplified by the queue sizes when used in decision making, i.e., (\ref{eq:Psi}). 
% 

%Doing so A continuous update version of $\mathtt{LRAM}$. In this case, the algorithm will first learn $\hat{\bv{r}}$. Then it runs the algorithm and continuously update $\hat{\bv{r}}$ and use that to make decisions.=

We now present the performance results for $\mathtt{DRAM}$. %\footnote{Even in the case when $\delta_z$}
% in the following theorem. 
\begin{theorem}\label{theorem:dram}
Suppose $\max(T_{\delta_r}, T_{\delta_z})<\infty$ with probability $1$. Suppose $g(\bv{\alpha})$ is polyhedral with $\rho=\Theta(1)>0$, and that $\delta_z\leq\epsilon_z$ and $\delta_r\leq\epsilon_r$, and $\bv{\alpha}^*+\bv{\theta}\succ\bv{0}$. Then, with a sufficiently large $V$, we have with probability $P_{\delta_z}P_{\delta_r}$ that, 
under $\mathtt{DRAM}$, 
\begin{eqnarray}
f_{\text{av}}^{\mathtt{DRAM}} \geq f_{\text{av}}^* - \frac{G}{V} - O(1/V+\delta_r). \label{eq:utility-per-dram}
\end{eqnarray}
Also, the fraction of time dropping happens is $O(V^{-\log(V)})$. Moreover, for all  queues, there exist $\Theta(1)$ constants $D, K, a$ such that: 
\begin{eqnarray}
\prob{ d_n(t)\geq \frac{3}{2}\zeta +D+\nu } &\leq& ae^{-K\nu} \label{eq:dram-d-bdd}\\
\prob{ Q_n(t)\geq \frac{3}{2}\zeta +D+\nu } &\leq& ae^{-K\nu} \label{eq:dram-q-bdd}\\
\prob{ H_m(t)\geq \frac{3}{2}\zeta +D+\nu } &\leq& ae^{-K\nu}. \label{eq:dram-h-bdd}
\end{eqnarray}
Thus, all queues are stable. $\Diamond$
\end{theorem}
\begin{proof}
%Omitted due to space limit. Please see \cite{huang-matching-14}.
See Appendix D. 
\end{proof}

Note that if we have $\delta_r=0$ with $T_{\delta_r}=0$ and $P_{\delta_r}=1$, then Theorem \ref{theorem:lram} recovers the known $[O(1/V), O(V)]$ utility-delay tradeoff for stochastic network problems \cite{neelynowbook}. On the other hand, if we also have $\delta_z=0$ with $T_{\delta_z}=0$ and $P_{\delta_z}=1$, then $\mathtt{DRAM}$ provides a new way for achieving the near-optimal $[O(1/V), O(\log(V)^2)]$ utility-delay tradeoff. 

In both $\mathtt{LRAM}$ and $\mathtt{DRAM}$,  it is possible to  continuously update the reward function estimations during the control steps. However,  this does not automatically guarantee that we can always eliminate the effect of $\delta_r$. This is so because the initial estimation $\hat{\bv{r}}$ may affect what options will be continuously updated later. On the other hand, the same performance results will hold if further updates do not increase $\delta_r$. 

%We remark here that  it is possible and common to continuously update the estimation during the matching phase. Intuitively this will perform better if ==the resulting algorithm does not exclude some options, also in a dynamic environment, it may be desirable to run under a slightly ``suboptimal'' condition==

\subsection{Convergence time}
Here we look at another important performance metric - algorithm convergence time, which  characterizes the time it takes for the algorithm to enter the ``steady state.'' Faster convergence implies better robustness against system statistics changes and higher efficiency in  resource allocation, and is particularly important when system statistics can change.  
The formal definition of convergence time is as follows  \cite{huang-learning-sig-14}. 
\begin{definition} \label{definition:convergence}
For a given constant $D$, the $D$-convergence time of a control algorithm $\Pi$, denoted by $T^{\Pi}_{D}$, is the time it takes for the queue vector $(\bv{d}(t), \bv{Q}(t), \bv{H}(t))$ ($(\hat{\bv{d}}(t), \hat{\bv{Q}}(t), \hat{\bv{H}}(t))$ under $\mathtt{DRAM}$) to get to within $D$ distance of $\bv{\alpha}^*+ \bv{\theta}$, i.e.,
\begin{eqnarray}
T_{D}^{\Pi}\triangleq\inf\{t\,|\, ||(\bv{d}(t), \bv{Q}(t), \bv{H}(t))-(\bv{\alpha}^*+\bv{\theta})||\leq D\}.\,\,\,\Diamond\label{eq:con-time}
\end{eqnarray}
\end{definition}

This definition of convergence time concerns about when an algorithm enters its ``optimal state.'' It is different from the metrics considered in \cite{li-convergence-13} and \cite{neely-convergence-15}, which concern about the time it takes for the objective value and constraints to be within certain accuracy. 
With Definition \ref{definition:convergence}, we have the following results: 
\begin{theorem}\label{theorem:convergence-time}
Suppose $g(\bv{\alpha})$ is polyhedral with $\rho=\Theta(1)>0$,   $\delta_z\leq\epsilon_z$ and $\delta_r\leq\epsilon_r$, and $\bv{\alpha}^*+\bv{\theta}\succ\bv{0}$. Then, with a sufficiently large $V$, we have: 
\begin{eqnarray}
\expect{T^{\mathtt{LRAM}}_{D_1}} & =& O(T_{\delta_r} + \Theta(V))\,\,w.p.\,\,P_{\delta_{r}} \label{eq:conv-lram}\\
\expect{T^{\mathtt{DRAM}}_{D_2}} & =& O((T_l  +  \Theta(\delta_{z}V))\,\,w.p.\,\,P_{\delta_{r}}P_{\delta_{z}}.  \label{eq:conv-dram}
\end{eqnarray}
Here $T_l\triangleq\max(T_{\delta_r}, T_{\delta_z})$ denotes the total learning time in $\mathtt{DRAM}$, $D_1\triangleq\Theta(\delta_{r}V)+\Theta(1)$, and $D_2\triangleq\Theta(\delta_{r}V)+\Theta(1)$.  $\Diamond$
\end{theorem}
\begin{proof}
%Omitted due to space limit. Please see \cite{huang-matching-14}. 
See Appendix E. 
\end{proof}

Combining Theorems \ref{theorem:lram}, \ref{theorem:dram}, and  \ref{theorem:convergence-time}, we see that $\delta_r$ largely affects the overall utility performance (reflected by $D_1$ and $D_2$), while $\delta_z$ can greatly improve the convergence time and delay! This indicates that information of different system components can play very different roles in algorithm performance and learning accuracies should be carefully chosen for meeting a desired performance goal. 

%learning capacities should be carefully allocated in order to meet a desired in system control. 

%\textbf{Remark:} It is important to 

%\subsection{WHy }

\subsection{Necessity in Controlling $\delta_r$}\label{subsection:2-queue-system}
Here we provide a simple example showing that it is necessary to control $\delta_r$ for good utility performance.  Hence, it is important to learn the reward value for each matching option. 
Consider the case when $N=2$ and $M=1$. Suppose $e(t)=A_1(t)=A_2(t)=1$ for all $t$. Also suppose $\bv{b}(t)\in\{(0, 1), (1, 0), (0, 0)\}$, that is, at every time $t$, we can only allocate resource to one or zero queue. 
Suppose $c(t)=0$, $\mu_n(t)=b_n(t)$, and $r_n(t) =\tilde{\mu}_n(t)$. Finally, assume that $U_1(r_1)=\log(1+r_1)$ and $\log(1+2r_2)$. 

In this case, the true optimal takes place at $\overline{r}_1=1/4$ and $\overline{r}_2=3/4$ with  $U_{\text{total}}=0.7828$. Now suppose we incorrectly estimate the rewards to be $r_n(t) =(1+\delta_{n})\tilde{\mu}_n(t)$. Then, one can show that the optimal rewards become:  
\begin{eqnarray}
r_1&=&\frac{2(1+\delta_1) (1+\delta_2) - 2(1+\delta_2) +(1+\delta_1)}{ 4(1+\delta_1) (1+\delta_2)  }\\ % \approx \frac{1}{4} + \frac{2\delta_1-\delta_2}{4}\\
r_2&=&\frac{ 2(1+\delta_1) (1+\delta_2) + 2(1+\delta_2) -(1+\delta_1)}{ 4(1+\delta_1) (1+\delta_2)  }, %\approx \frac{3}{4} - \frac{2\delta_1-\delta_2}{4},  
%r_1=\frac{1-2\delta}{4+4\delta}\approx \frac{1}{4} - \frac{\delta}{4}, \, r_2=\frac{3+2\delta}{4+4\delta}\approx \frac{3}{4}+  \frac{\delta}{4},
\end{eqnarray}
which is roughly $r_1\approx \frac{1}{4} + \frac{2\delta_1-\delta_2}{4}$ and $r_2\approx \frac{3}{4} - \frac{2\delta_1-\delta_2}{4}$. Thus, the resulting optimal utility is given by: 
\begin{eqnarray}
U_{\text{total}} \approx 0.7828 - \frac{|2\delta_1| +|\delta_2|}{5}. 
\end{eqnarray}
Therefore, in order to obtain an $O(1/V)$ close-to-optimal utility, it is necessary to ensure that $\max(|\delta_1|, |\delta_2|)=O(1/V)$. 

%However, if we incorrectly estimates $\bv{r}$, e.g., $r_1(t) =a_1 \tilde{\mu}_1(t)= (1+\delta_{r1})\tilde{\mu}_1(t)$ and $r_2(t) =a_2 \tilde{\mu}_2(t)=(1+\delta_{r2})\tilde{\mu}_2(t)$, then the optimal becomes: 
%\begin{eqnarray}
%r_1= \frac{2a_1a_2+2a_2-a_1}{4a_1a_2}, \,\,r_2= \frac{2a_1a_2-2a_2+a_1}{4a_1a_2}. 
%\end{eqnarray}

%In the ideal case, if we optimize the allocation and control, it can be shown that the optimal utility is achieved when $r_1=$ and $r_2=$. 

\section{Simulation} \label{section:simulation}
In this section, we present simulation results for our algorithms. We consider a system that has $N=2$ and $M=1$. We assume that $e(t)$ is $0$ or $2$ with equal probabilities.  $A_1(t) $ is $0$ or $2$ with equal probabilities and $A_2(t)$ is $1$ or $2$ with equal probabilities. $\bv{b}(t)\in\{(0, 1), (1, 0), (0, 0)\}$, i.e., at every time $t$, we  allocate one unit  resource to one or zero queue. We set $c(t) = b_1(t)+b_2(t)$ and $\mu_n(t)=b_n(t)$, and $\gamma_n(t)\in\{0, 1, 2\}$. 
There are two system states $\omega(t)\in\{1, 2\}$. In each state, the reward functions are given by $r_n(t) =w_n(\omega(t))\tilde{\mu}_n(t)$, where $w_1(1)=0.8$ and $w_2(1)=1$ and $w_1(2)=1$ and $w_2(2)=0.8$. Thus, the system state indicates which tasks are preferred under the specific condition. 
Every time the corresponding reward is either  $0.5r_n(t)$ or $1.5r_n(t)$, with equal probabilities. 
Finally, we assume that $U_1(r_1)=1.2\log(1+2r_1)$ and $U_2(r_2)=1.2\log(1+4r_2)$. 

From the definitions, we have that $\beta=5$, $\beta_{\mu}^u=\beta_{\mu}^l=1$ and $\beta_{\hat{r}} = \max_{\omega(t), n} \hat{r}_n(\omega(t))$. We also set $r_{\max}=2$, $\mu_{\max}=1$ and $b_{\max}=1$, $h_{\max}=2$. 
We simulate both $\mathtt{LRAM}$ and $\mathtt{DRAM}$ with $V=\{10 ,20, 50, 80, 100\}$. According to (\ref{eq:theta-value1}) and (\ref{eq:theta-value2}), $\theta_1=(5V+2)\beta_{\hat{r}} +4$ and $\theta_2 = (5V +2)\beta_{\hat{r}} +3$. 
In $\mathtt{DRAM}$, we set $\zeta=\log(V)^2$. We use $\mathtt{TBS}(s_{th})$ to estimate $\bv{r}$ and set $s_{th}=\log(V)^2$, and use $\mathtt{TLS}(T_{th})$ to estimate $\bv{\pi}$ and set $T_{th}=T_{tbs}$. 
%the sampling-based estimator and requires $T_{\delta_r}$ to be the time when each $r$ value is sampled as least $\log(V)^2$ times. 
In $\mathtt{LRAM}$, we randomly add or subtract the estimation error $\delta_r$ from the true values. 
\begin{figure}[cht]
\vspace{-.08in}
\begin{center}
\includegraphics[height=1.6in, width=3.2in]{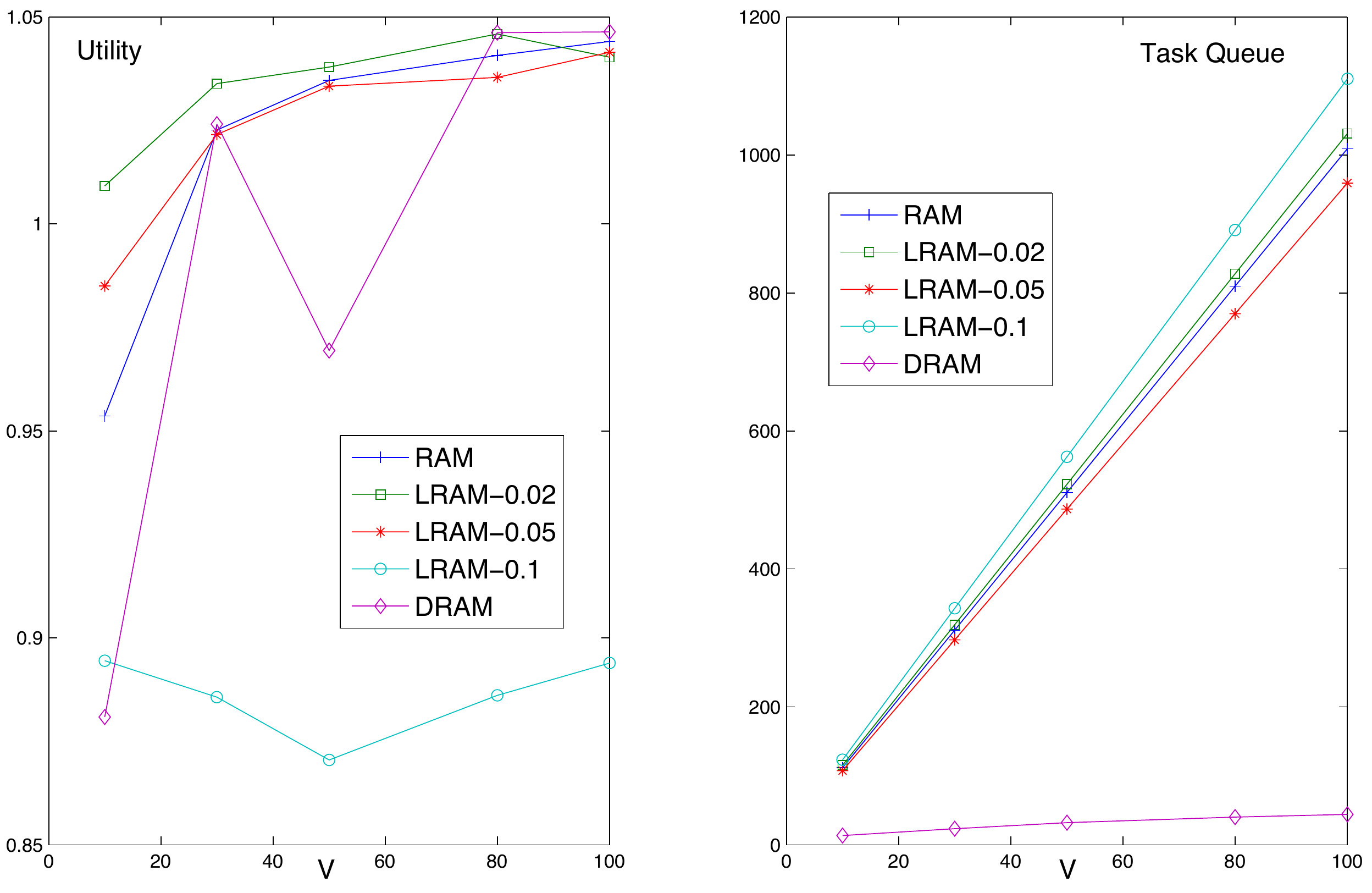}
\vspace{-.1in}
\end{center}
\vspace{-.1in}
\caption{Utility performance and task queue size. }\label{fig:utility}
\vspace{-.1in}
\end{figure}

Fig. \ref{fig:utility} first shows the utility performance and the task queue behavior of $\mathtt{LRAM}$ and $\mathtt{DRAM}$, where the number after $\mathtt{LRAM}$ denotes $\delta_r$. We  see from the left plot that except for $\delta_r=0.1$, $\mathtt{LRAM}$ performs very well under all other error values, suggesting that estimation error indeed plays an important role in system utility. We also see that $\mathtt{DRAM}$ performs very well starting from $V\geq 50$. The right plot shows the backlog (delay) performance under different schemes. It is evident that $\mathtt{DRAM}$ achieves an $O(\log(V)^2)$ delay in this case, while all other variants possess an $O(V)$ delay. This demonstrates the importance of incorporating system dynamics information into algorithm design.

Fig. \ref{fig:queue} then shows the resource queue  $H(t)$ and  deficit queues $\bv{d}(t)$. We see  that $\mathtt{DRAM}$ ensures an $O(\log(V)^2)$ average resource queue, while other algorithms result in an $O(V)$ queue size.  This implies that $\mathtt{DRAM}$ ensures a very short stay in the system for the  resource items! This feature is  particularly useful if the resource items are human users, e.g., in crowdsourcing. 
\begin{figure}[cht]
%\vspace{-.08in}
\begin{center}
\includegraphics[height=1.6in, width=3.2in]{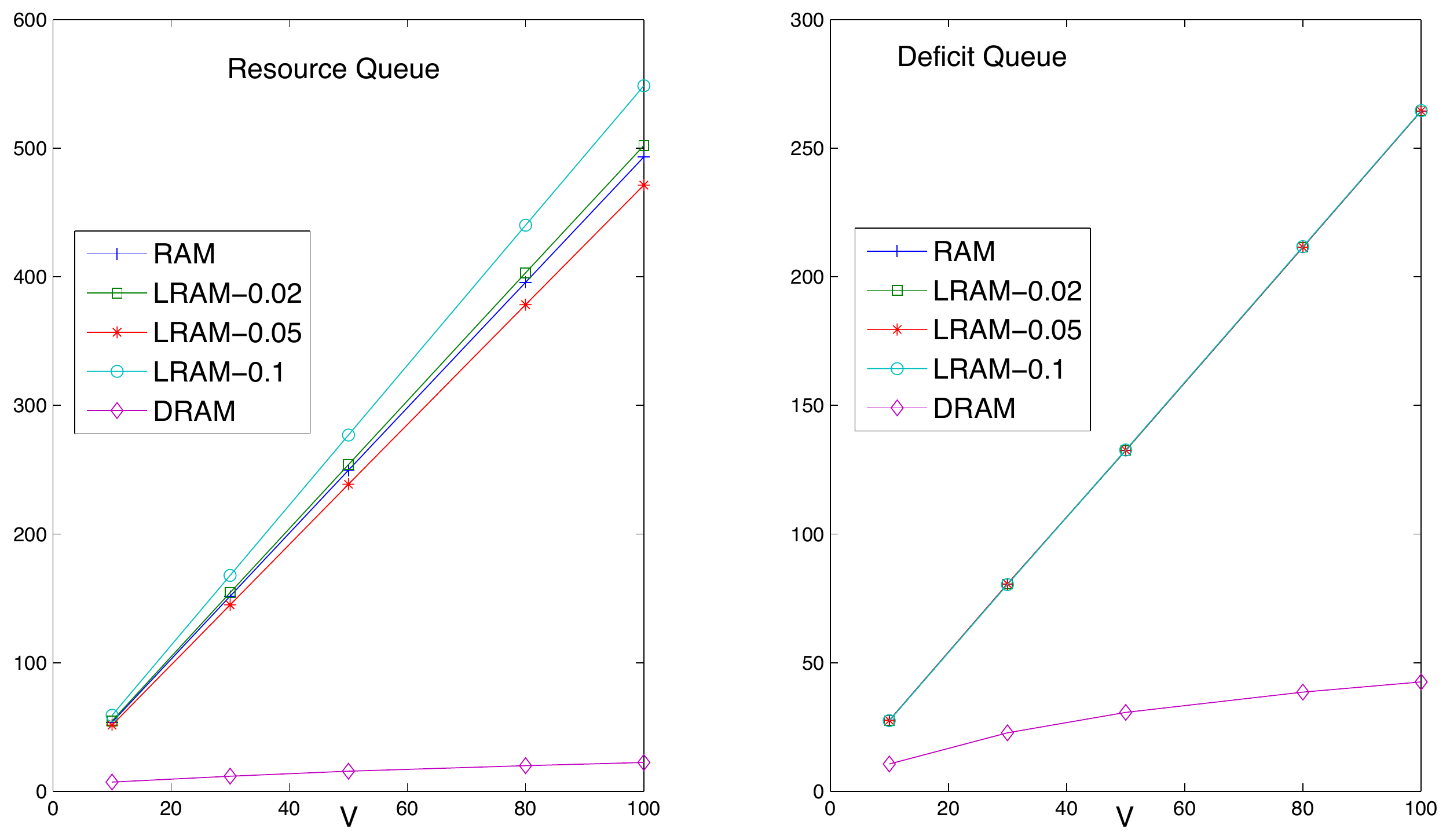}
\vspace{-.1in}
\end{center}
\vspace{-.1in}
\caption{Resource queue and deficit queue sizes. }\label{fig:queue}
\vspace{-.1in}
\end{figure}

%=simulate the continuous learning one=

Finally, Fig. \ref{fig:converge} shows the convergence behavior of the algorithms for $V=100$. Here we show the resource queue value as its convergence time dominates the others. We see that $\mathtt{RAM}$ takes an $O(V)$ time to converge, which is expected. We also observe that $H(t)$ under $\mathtt{LRAM}$-$0.05$ and $\mathtt{LRAM}$-$0.1$ converge to values slightly above those under $\mathtt{RAM}$. This explains why their performance is slightly worse. 
On the other hand, we also see that $\mathtt{DRAM}$ converges quickly. The reason its steady state value is slightly above that under $\mathtt{RAM}$ is due to the inaccuracy of $\hat{\bv{r}}$. Even in this case, we see that there is a $2.5\times$ convergence speedup (most of the learning time is due to sampling) and $\mathtt{DRAM}$ achieves very good performance. In the case when $\bv{r}$ can be obtained from other data source beforehand, which can commonly be done in practice, e.g., in online advertising, we see that $\mathtt{DRAM}$ (called state-only $\mathtt{DRAM}$ in this case) achieves a $10\times$ convergence speedup ($50$ slots vs. $500$ slots)! 
\begin{figure}[cht]
\vspace{-.08in}
\begin{center}
\includegraphics[height=1.5in, width=3.2in]{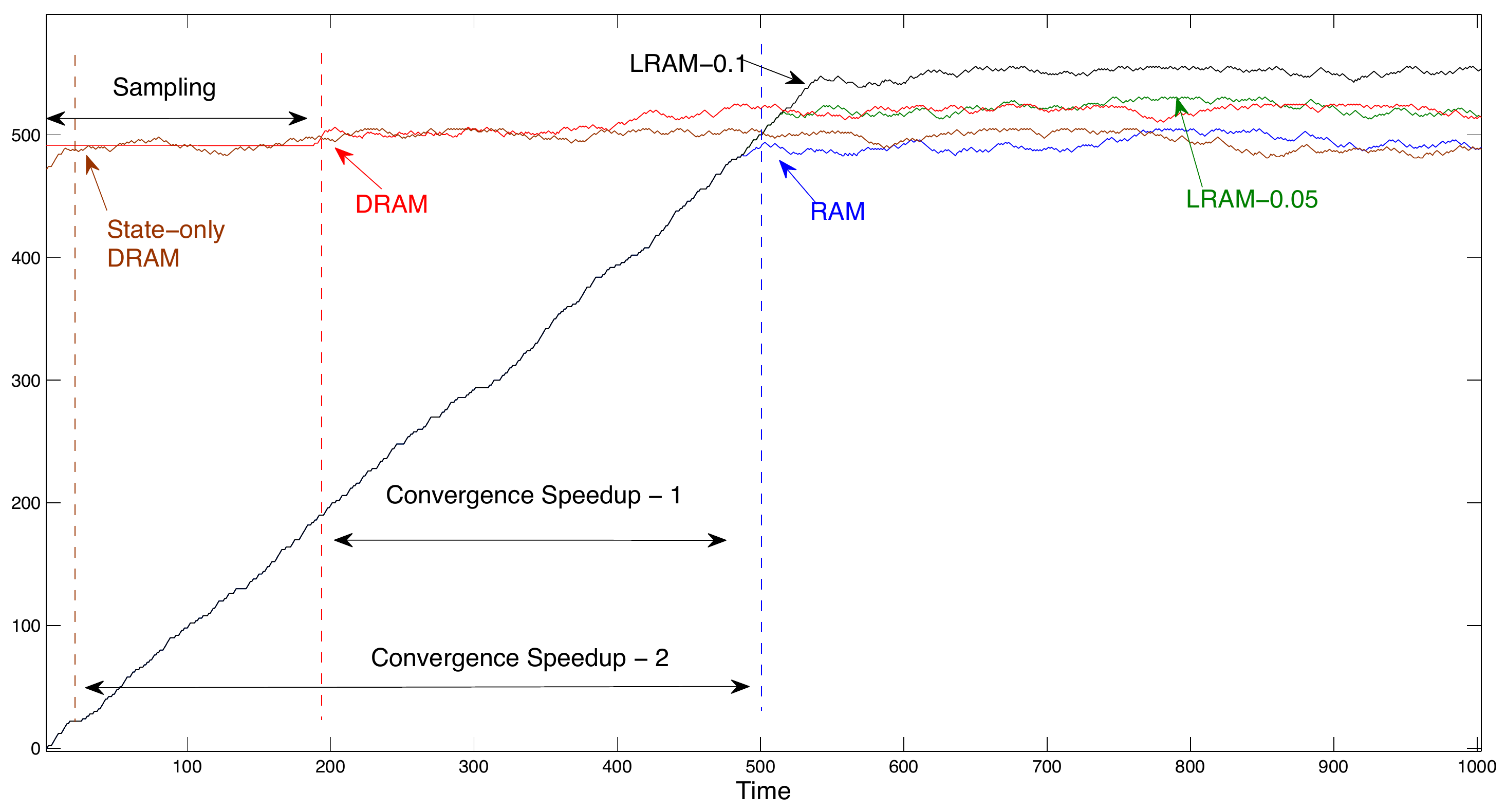}
\vspace{-.1in}
\end{center}
\vspace{-.1in}
\caption{Convergence of $\mathtt{LRAM}$ and $\mathtt{DRAM}$ for $V=100$. }\label{fig:converge}
\vspace{-.1in}
\end{figure}

We observe in all simulation instances that no dropping occurs. This demonstrates the effectiveness of the algorithms and validates Theorem \ref{theorem:dram}. 
%=LIFO?=

%=Add a plot that shows quick adapt when distribution changes? the $V^{1-c/2}$ should be good=

\section{Conclusion}\label{section:conclusion}
In this paper, we study the problem of optimal matching  with queues in dynamic systems. We develop two online learning-aided algorithms $\mathtt{LRAM}$ and $\mathtt{DRAM}$ for resolving the challenging underflow problem and to achieve near-optimality. We show that $\mathtt{LRAM}$  achieves an $O(\epsilon+\delta_r)$ system utility, for any $\epsilon>0$, while ensuring an $O(1/\epsilon)$ delay bound and an $O(1/\epsilon)$ algorithm convergence time. $\mathtt{DRAM}$, on the other hand,  guarantees a similar $O(\epsilon+\delta_r)$ system utility, while  achieving an $O(\delta_{z}/\epsilon + \log(1/\epsilon)^2)$ delay bound and an $O(\delta_z/\epsilon)$ algorithm convergence time, which can be significantly better compared to $\mathtt{LRAM}$ when $\delta_z$ is small. Our algorithms and results reveal the interesting fact that different system information can play very different roles in algorithm performance and provide insights into joint learning-control algorithm design for dynamic systems. 
%it is important to fully understand the value of information in such systems. 

$\vspace{-.2in}$
\bibliographystyle{unsrt}
\bibliography{mybib}

\section*{Appendix A -- Proof of Lemma \ref{lemma:drift}}
We prove Lemma \ref{lemma:drift} here. 
\begin{proof} (Lemma \ref{lemma:drift}) 
Using the queueing dynamics (\ref{eq:q-dyn}) and (\ref{eq:h-dyn}), we have: 
\begin{eqnarray*}
\hspace{-.3in}&&(Q_n(t+1)-\theta_1)^2\leq (Q_n(t)-\theta_1)^2  \\
\hspace{-.3in}&&\qquad\quad- 2(Q_n(t)-\theta_1)(\mu_n(t)-R_n(t)) +R_n(t)^2+\mu_n(t)^2. 
\end{eqnarray*}
Similarly, we get: 
\begin{eqnarray*}
\hspace{-.3in}&&(H_m(t+1)-\theta_2)^2\leq (H_m(t)-\theta_2)^2  \\
\hspace{-.3in}&& - 2(H_m(t)-\theta_2)(\sum_{n}b_{mn}(t)-h_m(t)) +(\sum_{n}b_{mn}(t))^2+h_m(t)^2,  
\end{eqnarray*}
and that  
\begin{eqnarray*}
\hspace{-.3in}&&d_n(t+1)^2\leq d_n(t)^2  - 2d_n(t)(\kappa_n(t)-\gamma_n(t)) +\kappa_n(t)^2+\gamma_n(t)^2. 
\end{eqnarray*}
Summing the above and using the definition of $L(t)$ and $\Delta_V(t)$, we get: 
\begin{eqnarray*}
&&L(t+1)-L(t)-Vf(t)\\
&&\qquad \leq G - V(\sum_nU_n(\gamma_n(t))-c(t)) \\
&&\qquad\qquad - \sum_n (Q_n(t)-\theta_1)(\mu_n(t)-R_n(t))\\
&& \qquad\qquad- \sum_{n} d_n(t)(\kappa_n(t)-\gamma_n(t)) \\
&&\qquad \qquad-\sum_m(H_m(t)-\theta_2)(\sum_{n}b_{mn}(t)-h_m(t)). 
\end{eqnarray*}
Here $G\triangleq N(A_{\max}^2 +\mu_{\max}^2+2r_{\max}^2)+Mh_{\max}^2+MN^2b_{\max}^2$. 
Rearranging terms, we obtain: 
\begin{eqnarray*}
\hspace{-.3in}&&L(t+1)-L(t)-Vf(t)\\
\hspace{-.3in}&&\quad\leq G - V\sum_n[U_n(\gamma_n(t)) -d_n(t)\gamma_n(t)] \\ 
\hspace{-.3in}&&\quad\qquad   + \sum_n (Q_n(t)-\theta_1)R_n(t)  + \sum_m(H_m(t)-\theta_2)h_m(t) \\
\hspace{-.3in}&& \quad \qquad  +\bigg[Vc(t) - \sum_m(H_m(t)-\theta_2)\sum_{n}b_{mn}(t) \\
\hspace{-.3in}&&\quad\qquad\quad  -\sum_n(Q_n(t)-\theta_1)\mu_n(t)  -  \sum_{n} d_n(t)\kappa_n(t)\bigg]. 
\end{eqnarray*}
Taking an expectations on both sides conditioning on $\bv{y}(t)$ and using the fact that $\kappa_n(t)$ is an i.i.d. random variable given $\bv{z}(t), \bv{b}(t), \bv{Q}(t)$, we see that the lemma follows. 
\end{proof}
\section*{Appendix B -- Proof of Lemma \ref{lemma:module}}
We prove Lemma \ref{lemma:module} here. 
In our proof, we will use the following theorem from \cite{chung_concentration}. 
\begin{theorem}\label{eq:gen-concentration}
\cite{chung_concentration} Suppose $X_i$ are independent random variables satisfying $X_i\leq B$ for $1\leq i\leq n$. Let $X=\sum_iX_i$ and $\|X\| = \sqrt{\sum_i\expect{X_i^2}}$. Then we have: 
\begin{eqnarray}
\prob{X\leq \expect{X} -b} \leq e^{  -\frac{b^2}{2(\|X\|^2 + Bb/3)}  }. \Diamond
\end{eqnarray}
\end{theorem}

\begin{proof} (Lemma \ref{lemma:module}) 
First of all, we show that $\expect{T_{tbs}} = \Theta(\log(V)^2)$. To see this, notice that since each $\script{B}_{k}$ is finite, if the state $\bv{z}(t)=k$ appears $|\script{B}_{k}|s_{th}$ times, we must have sampled every $\bv{b}\in \script{B}_{k}$ $s_{th}$ times. Thus, 
\begin{eqnarray}
T_{tbs}\leq \sum_{k} T\{\text{visit}\,\bv{z}_k\,\,|\script{B}_{k}|s_{th}\,\,\text{times}\}. 
\end{eqnarray}
Taking the expectations, we see that $\expect{T_{tbs}}\leq s_{th}\sum_{k}\frac{|\script{B}_{k}|}{\pi_{k}}$. 
%\end{eqnarray}

Then, we see that in $\hat{\bv{r}}$, each single value has been sampled $s_{th}$ times. Using Theorem \ref{eq:gen-concentration} and the fact that $\kappa_n(t)\leq r_{\max}$, we get: 
\begin{eqnarray*}
&&\prob{\sum_{t=1}^{T_{tbs}}1_{[\bv{z}(t)=\bv{z}, \bv{b}(t)=\bv{b}]} \kappa_n(t)   \leq r_n(\bv{z}, \bv{b})s(\bv{z}, \bv{b}, T_{tbs}) - b} \\
&&\leq e^{   - \frac{ b^2}{  2(s_{th}r_{\max}^2 + r_{\max}b/3)}   }. 
\end{eqnarray*}
Choosing $b=\sqrt{s_{th}}\log(s_{th})$, dividing both sides of the inequality inside by $s(\bv{z}, \bv{b}, T_{tbs})$, and using $s(\bv{z}, \bv{b}, T_{tbs})\geq s_{th}$, we get: 
\begin{eqnarray*}
\prob{  \hat{r}_n(\bv{z}, \bv{b})  \leq r_n(\bv{z}, \bv{b}) - \frac{\log(s_{th})}{\sqrt{s_{th}}}}  \leq e^{   - \frac{ s_{th}\log(s_{th})^2}{  2(r_{\max}^2s_{th} + r_{\max}\sqrt{s_{th}}/3)}   }. 
\end{eqnarray*}
Using Theorem \ref{eq:gen-concentration} with $-X$, we get a similar bound for the other side. Hence, 
\begin{eqnarray*}
\prob{  |\hat{r}_n(\bv{z}, \bv{b}) - r_n(\bv{z}, \bv{b})| \leq \frac{\log(s_{th})}{\sqrt{s_{th}}}}  \leq 2e^{   - \frac{ s_{th}\log(s_{th})^2}{  2(r_{\max}^2s_{th} + r_{\max}\sqrt{s_{th}}/3)}   }. 
\end{eqnarray*}
Using the union bound, we see that Part (a) follows. Part (b) can be proven similarly. 
\end{proof}

\section*{Appendix C -- Proof of Theorem \ref{theorem:lram}}
We present the proof for Theorem \ref{theorem:lram} here. For our analysis, we will use the following result, which is Theorem $1$ in \cite{huangneely_qlamarkovian}. 
\begin{theorem}\label{theorem:dual-fav} \cite{huangneely_qlamarkovian} 
Let $\bv{\alpha}^{*}$ be an optimal solution of (\ref{eq:dual-fun}). Then, $g(\bv{\alpha}^{*})\geq Vf_{\text{av}}^*$. $\Diamond$
\end{theorem}

%\section*{Appendix B -- Proof of Lemma \ref{lemma:q-bdd}}
\begin{proof} (Theorem \ref{theorem:lram}) 
(\textbf{Queueing}) %We first prove the queueing bounds. 
First consider $d_n(t)$. We see that if $d_n(t)\leq V\beta$, then $d_n(t+1)\leq V\beta+r_{\max}$. On the other hand, from (\ref{eq:quota}), whenever $d_n(t)>V\beta$, $\gamma_n(t)=0$. Hence, $d_n(t)$ will not further increase. This proves the bound for $d_n(t)$. 

The bounds for $Q_n(t)$ and $H_m(t)$ can be similarly proven by noticing that $\mathtt{LRAM}$ will not further admit tasks once $Q_n(t)\geq\theta_1$ and it  does not admit resources once $H_m(t)\geq\theta_2$. 
%\end{proof}

(\textbf{Utility}) 
We carry out the proof by comparing the RHS of (\ref{eq:drift}) under $\mathtt{LRAM}$ with any other control policy, including the ones that do not respect the no-underflow constraint (\ref{eq:no-underflow}). 

% =what about $\tilde{\mu}$=. 

To this end,  look at (\ref{eq:Psi}). We want to show that even without constraint (\ref{eq:no-underflow}), $\mathtt{LRAM}$ ensures that (i) whenever $H_m(t)\leq Nb_{\max}$, $b_{mn}(t)=0$ for all $n$, and (ii) whenever $Q_n(t)\leq \mu_{\max}$, $\mu_n(t)=0$. 
We first show (i). Suppose $H_m(t)< Nb_{\max}$. We see from (\ref{eq:theta-value2}) that: 
\begin{eqnarray}
H_m(t) - \theta_2< - (V\beta+r_{\max})\beta_{\hat{r}} -  r_{\max}\beta_{\mu}^u. 
\end{eqnarray}
In this case, denote the optimal resource allocation vector as $\bv{b}^*$ and suppose there is one $n$ with $b_{mn}^*>0$. Let $\tilde{\bv{b}}$ be the vector obtained from $\bv{b}^*$ by setting $b_{mn}^*=0$. We have: 
\begin{eqnarray}
\hspace{-.3in}&&\Psi_{\hat{r}}(\bv{b}^*) - \Psi_{\hat{r}}(\tilde{\bv{b}}) \label{eq:contradict1}\\
\hspace{-.3in}&&= Vc(\bv{z}, \bv{b}^*) -  Vc(\bv{z}, \tilde{\bv{b}}) -   (H_m(t)-\theta_2)b^*_{mn}   \nonumber\\
\hspace{-.3in}&& \qquad - (Q_n(t)-\theta_1) (\mu_n(\bv{z}(t), \bv{b}^*) - \mu_n(\bv{z}(t), \tilde{\bv{b}})) \nonumber\\
\hspace{-.3in}&&\qquad  - d_n(t) (\hat{r}_n(\bv{z}(t), \tilde{\mu}_n(\bv{z}(t), \bv{b}^*))   -  \hat{r}_n(\bv{z}(t), \tilde{\mu}_n(\bv{z}(t), \tilde{\bv{b}}))) \nonumber\\
\hspace{-.3in}&&> [(V\beta+r_{\max})\beta_{\hat{r}} +  r_{\max}\beta^u_{\mu}] b^*_{mn} - r_{\max}\beta^u_{\mu}b^*_{mn} \nonumber\\
\hspace{-.3in}&&\qquad - (V\beta+r_{\max})\beta_{\hat{r}} b^*_{mn} =0. \nonumber
\end{eqnarray} 
In the inequality, we have used the fact that $Q_n(t)-\theta_1\leq r_{\max}$, $\mu_n(\bv{z}(t), \bv{b}^*) - \mu_n(\bv{z}(t), \tilde{\bv{b}})\leq\beta_{\mu}^u$, $d_n(t)\leq V\beta+r_{\max}$, and $\hat{r}_n(\bv{z}(t), \tilde{\mu}_n(\bv{z}(t), \bv{b}^*))   -  \hat{r}_n(\bv{z}(t), \tilde{\mu}_n(\bv{z}(t), \tilde{\bv{b}}))\leq\beta_{\hat{r}}b^*_{mn}$. 
However, (\ref{eq:contradict1}) contradicts with the fact that $\bv{b}^*$ is the minimizer of $\Psi_{\hat{r}}(\bv{b})$ and shows that we must have $b_{mn}^*=0$ $\forall\, n$ whenever $H_m(t)<Nb_{\max}$. 

Now we look at the case of $\bv{Q}(t)$. Suppose $Q_n(t)< \mu_{\max}$. Then, we have: 
\begin{eqnarray}
Q_n(t) - \theta_1 <  - (h_{\max}+ (V\beta+r_{\max})\beta_{\hat{r}})/\beta^l_{\mu}. 
\end{eqnarray}
We similarly let $\bv{b}^*$ be the optimal solution. Then, we construct $\tilde{\bv{b}}$ by setting $b_{m^*n}^*>0$ to zero, where $m^*=\arg\min_m b_{mn}^*$. In this case, if $\mu_n(\bv{z}(t), \bv{b}^*)=0$,  we are done. Otherwise, % we have: 
\begin{eqnarray}
\hspace{-.3in}&&\Psi_{\hat{r}}(\bv{b}^*) - \Psi_{\hat{r}}(\tilde{\bv{b}})\\
\hspace{-.3in}&&> - h_{\max} b^*_{m^*n} - (V\beta+r_{\max})\beta_{\hat{r}} b^*_{m^*n}\nonumber\\
\hspace{-.3in}&&\qquad + ((h_{\max}+ (V\beta+r_{\max})\beta_{\hat{r}})/\beta^l_{\mu}) \beta^l_{\mu}b^*_{m^*n}  =0.  \nonumber
\end{eqnarray}
The inequality follows since $H_m(t)-\theta_2\leq h_{\max}$, $d_n(t)\leq V\beta+r_{\max}$, and that $\mu_n(\bv{z}(t), \bv{b}^*)>0$, which implies $\mu_n(\bv{z}(t), \bv{b}^*)\geq \beta_{\mu}^l b^*_{m^*n}$
This contradicts with the fact that $\bv{b}^*$ is the minimizer  and shows that whenever $Q_n(t)< \mu_{\max}$, $\bv{b}^*_n=0$. 

These two properties imply that $\mathtt{LRAM}$ automatically guarantees that the no-underflow constraints are satisfied for all $H_m(t)$ and that we always have $\tilde{\mu}_n(t) = \min[Q_n(t), \mu_n(t)] = \mu_n(t)$. To compare our control policy with any other matching policies for the drift (\ref{eq:drift}), we still need to show that the actions under $\mathtt{LRAM}$, which is based on the estimated reward matrix $\hat{\bv{r}}$, ensure that the RHS of (\ref{eq:drift}), defined with the true reward $\bv{r}$, is approximately minimized. 

To do so, first observe that $\bv{\gamma}(t)$, $\bv{R}(t)$ and $\bv{h}(t)$ are optimally chosen given $\bv{y}(t)$ and $\bv{z}(t)$. Hence, the only approximation comes from choosing $\bv{b}(t)$. Let $\bv{b}_{\hat{r}}^*(t)$ be the chosen vector under $\mathtt{LRAM}$ and let $\bv{b}_r^*(t)$ be the vector chosen if $\bv{r}$ is used. We have: 
\begin{eqnarray}
\Psi_{\hat{r}}(\bv{b}_{\hat{r}}^*(t)) &\leq& \Psi_{\hat{r}}(\bv{b}_r^*(t))\\
&=& \Psi_{r}(\bv{b}_r^*(t)) + \sum_n d_n(t)[ \hat{r}_n(\bv{b}_r^*(t)) -  r_n(\bv{b}_r^*(t))]. \nonumber
\end{eqnarray}
On the other hand, we also have: 
\begin{eqnarray}
\Psi_{\hat{r}}(\bv{b}_{\hat{r}}^*(t)) = \Psi_{r}(\bv{b}_{\hat{r}}^*(t)) + \sum_n d_n(t)[ \hat{r}_n(\bv{b}_{\hat{r}}^*(t)) -  r_n(\bv{b}_{\hat{r}}^*(t))]. \nonumber
\end{eqnarray}
Combining the above two equalities and using the fact that $\Gamma_r$ is a $(T_{\delta_r}, P_{\delta_r}, \delta_r)$-learning module, we see that with probability $P_{\delta_r}$, 
\begin{eqnarray}
 \Psi_{r}(\bv{b}_{\hat{r}}^*(t)) &\leq& \Psi_{r}(\bv{b}_r^*(t)) + 2\sum_nd_n(t) \delta_r \nonumber\\
 &\leq&\Psi_{r}(\bv{b}_r^*(t)) + 2N(V\beta+r_{\max}) \delta_r. 
\end{eqnarray}
This shows that the RHS of (\ref{eq:drift}) under $\mathtt{LRAM}$ is minimized to within $2N(V\beta+r_{\max}) \delta_r$, over any other policies. Comparing this to $g_k(\bv{\alpha}^{d}, \bv{\alpha}^{q}, \bv{\alpha}^{h})$ in  (\ref{eq:dual-zk}) and using the definition of $g(\bv{\alpha}^{d}, \bv{\alpha}^{q}, \bv{\alpha}^{h})$, we conclude that: 
\begin{eqnarray}
\hspace{-.3in}\Delta_V(t) &=&\expect{L(t+1)-L(t) - Vf(t)\left.|\right. \bv{y}(t)}\nonumber \\
\hspace{-.3in} &\leq& G + 2N(V\beta+r_{\max}) \delta_r - g(\bv{d}(t), \bv{Q}(t), \bv{H}(t))\nonumber\\
\hspace{-.3in}&\stackrel{(a)}{\leq}& G+ 2N(V\beta+r_{\max}) \delta_r- Vf_{\text{av}}^*.
\end{eqnarray}
Here (a) follows from Theorem \ref{theorem:dual-fav}. 
Taking an expectation over $\bv{y}(t)$ on both sides and carrying out a telescoping sum from $t=0, ..., T-1$, and dividing both  sides by $VT$, we obtain: 
\begin{eqnarray*}
\frac{1}{T}\sum_{t=0}^{T-1}\expect{f(t)}& = &\frac{1}{T}\sum_{t=0}^{T-1}\sum_n\expect{U_n(\gamma_n(t)) -  c(t)}\\
&\geq& f_{\text{av}}^* - \frac{G+ r_{\max} \delta_r}{V} - 2N\beta\delta_r. 
\end{eqnarray*}
Taking a limit as $T\rightarrow\infty$, and using Jensen's inequality and the fact that $U_n(\overline{r}_n)$ is concave, we get: 
\begin{eqnarray}
\sum_nU_n(\overline{\gamma}_n) - \overline{c}\geq f_{\text{av}}^* - \frac{G+ r_{\max} \delta_r}{V} - 2N\beta\delta_r. 
\end{eqnarray}
Finally, using the fact that $\bv{d}(t)$ is bounded, which implies $\overline{\gamma}_n\leq \overline{r}_n$ for all $n$, and that $U_n(r)$ is increasing, we see that the theorem follows. 
%Using Theorem \ref{theorem:dual-fav}, we see that: 
%This completes the proof of the theorem. 
%However, since $\bv{b}(t)$ under 
% 
% 
% 
%
\end{proof}

\section*{Appendix D -- Proof of Theorem \ref{theorem:dram}}
Here we prove the performance of $\mathtt{DRAM}$. We  carry out the analysis in the following steps. First, we show that the estimated optimal multiplier $\hat{\bv{\alpha}}=(\hat{\bv{\alpha}}^{d*}, \hat{\bv{\alpha}}^{q*}, \hat{\bv{\alpha}}^{h*})$ is close to the true optimal. Then, we show via \emph{drift-augmentation} that  the definitions of $\hat{\bv{Q}}(t)$, $\hat{\bv{H}}(t)$, and $\hat{\bv{d}}(t)$ ensure a near-optimal  algorithm performance. 

We now have the following lemma for the first step. In the lemma, we denote $\tilde{\bv{\alpha}}^*$ the optimal solution for $\tilde{g}(\bv{\alpha})$, which is $g(\bv{\alpha})$ with $\bv{\pi}$ and $\hat{\bv{r}}$. 
\begin{lemma}\label{lemma:dist-bdd}
Suppose $g(\bv{\alpha})$ is polyhedral with $\rho=\Theta(1)>0$, and that $\delta_z\leq\epsilon_z$ and $\delta_r\leq\epsilon_r$. Then, with probability $P_{\delta_z}P_{\delta_r}$, we have: 
\begin{eqnarray}
\| \hat{\bv{\alpha}}^* - \tilde{\bv{\alpha}}^*\| &\leq& \frac{2\delta_z  Vf_{\max}\vartheta}{\rho} \label{eq:dist-bdd}\\
\| \bv{\alpha}^* - \tilde{\bv{\alpha}}^*\| &\leq& \frac{2Vf_{\max}\delta_r}{\rho\eta}, \label{eq:dist-bdd2}
\end{eqnarray}
where $\vartheta\triangleq |\script{Z}|(1+r_{\max} + A_{\max}+\mu_{\max} + Nb_{\max}+h_{\max})/\eta$ and $\eta=\Theta(1)$. $\Diamond$
\end{lemma}
\begin{proof}
See Appendix F. 
\end{proof}

In our analysis, we make use of the following two results. 
\begin{lemma}\label{lemma:q-avgrate} \cite{huang-learning-sig-14} 
Let $Q(t)$ be the size of a single queue with dynamics $Q(t+1)=[Q(t)-\mu(t)]^++A(t)$. Suppose $0\leq \mu(t), A(t)\leq\mu_{\max}=\Theta(1)$ for all $t$ and that the queue is stable. Then,  
\begin{eqnarray}
\overline{\mu(t)} - \overline{A(t)} \leq \mu_{\max}\prob{Q(t)<\mu_{\max}}. \label{eq:excessrate}
\end{eqnarray}
Here $\overline{x(t)}\triangleq \lim_{T\rightarrow\infty}\frac{1}{T}\sum_{t=0}^{T-1}\expect{x(t)}$. $\Diamond$ 
\end{lemma}

\begin{theorem} \label{theorem:attraction} 
%Suppose there exists a constant $\epsilon_{\beta}=\Theta(1)>0$, such that for any $\hat{\bv{\beta}}$ with $||  \hat{\bv{\beta}} - \bv{\beta}  ||\leq \epsilon_{\beta}$, $g_{\bv{\pi}}(\bv{h}, \bv{q}, \hat{\bv{\beta}})$  is  polyhedral with $\rho = \Theta(1)>0$, and 
Under  $\mathtt{LRAM}$ with reward functions $\hat{\bv{r}}$, there exist $\Theta(1)$ constants $a$, $K$, and $D$, such that,  
\begin{eqnarray}
\script{P}(D, \nu)\leq ae^{-K\nu},\label{eq:pm_ineq}
\end{eqnarray}
where $\script{P}(D, \nu)$ is defined as:
\begin{eqnarray}
\hspace{-.3in}&&\script{P}(D, \nu)\triangleq\lim_{t\rightarrow\infty}\prob{\| (\bv{d}(t), \bv{Q}(t), \bv{H}(t))  - (\tilde{\bv{\alpha}}^*+\bv{\theta}) \|>D+\nu}. \Diamond\nonumber 
\end{eqnarray}
\end{theorem}
\begin{proof}
Similar to the proof of Theorem 1 in \cite{huangneely_dr_tac}. Omitted for brevity. 
\end{proof}

\begin{proof} (Theorem \ref{theorem:dram}) 
To prove Theorem \ref{theorem:dram}, we define the following drift-augmentation term: 
\begin{eqnarray}
\hspace{-.3in}&&\Delta_a(t)\triangleq - \sum_n (\hat{\alpha}^{q*}_n - \zeta)\expect{\mu_n(t)-R_n(t)\left.|\right. \bv{y}(t)}  \\
&& \qquad\qquad- \sum_{n} (\hat{\alpha}^{d*}_n - \zeta)\expect{r_n(t)-\gamma_n(t)\left.|\right. \bv{y}(t)}   \nonumber\\
&&\qquad \qquad-\sum_m(\hat{\alpha}^{h*}_m - \zeta)\expect{\sum_{n}b_{mn}(t)-h_m(t)\left.|\right. \bv{y}(t)}.  \nonumber
\end{eqnarray}
Adding it to both sides of (\ref{eq:drift}), we get: 
\begin{eqnarray}
\hspace{-.3in}&&\Delta_V(t)+\Delta_a(t)\label{eq:aug-drift-2}\\
\hspace{-.3in}&&\leq G - V\sum_n\expect{U_n(\gamma_n(t)) -\hat{d}_n(t)\gamma_n(t)\left.|\right. \bv{y}(t) } \nonumber\\ 
\hspace{-.3in}&&\qquad\qquad\quad + \sum_n (\hat{Q}_n(t)-\theta_1)\expect{R_n(t)\left.|\right. \bv{y}(t)}  \nonumber\\
\hspace{-.3in}&&\qquad\qquad\quad + \sum_m(\hat{H}_m(t)-\theta_2)\expect{h_m(t)\left.|\right. \bv{y}(t)}  \nonumber\\
\hspace{-.3in}&&  \qquad\qquad   +\,\expect{Vc(t) -  \sum_m(\hat{H}_m(t)-\theta_2)\sum_{n}b_{mn}(t)  \nonumber \\
\hspace{-.3in}&&\qquad\qquad  -\sum_n(\hat{Q}_n(t)-\theta_1) \mu_n(t)-  \sum_{n} \hat{d}_n(t)r_n(t)\left.|\right. \bv{y}(t)}.  \nonumber
\end{eqnarray}
Note that (\ref{eq:aug-drift-2}) also holds under our dropping action, because it is equivalent to modifying the dynamics of $\bv{H}(t)$ to $H_m(t+1) = (H_m(t) -\sum_{n}(t)b_{mn}(t))^++h_m(t)$. This is important,  for it allows us to analyze the performance with the drift analysis. 

Using Lemma \ref{lemma:dist-bdd}, we know that with probability $P_{\delta_r}P_{\delta_z}$, $\|\tilde{\bv{\alpha}}^* - \hat{\bv{\alpha}}\|\leq  \frac{2\delta_z  Vf_{\max}\vartheta}{\rho}$. Using $\zeta \triangleq 2\max(\delta_{z} V\log(V)^2, \log(V)^2)$, we see that when $V$ is large, we have: 
\begin{eqnarray}
\frac{2\delta_z  Vf_{\max}\vartheta}{\rho}\leq \zeta/2. 
\end{eqnarray}
Thus, 
\begin{eqnarray}
   \tilde{\bv{\alpha}}^* - \frac{3}{2}\bv{\zeta} \leq  \hat{\bv{\alpha}}^* - \bv{\zeta} \leq\tilde{\bv{\alpha}}^* - \frac{1}{2}\bv{\zeta}, 
\end{eqnarray}
where the inequality is taken entry-wise. This implies that with probability $P_{\delta_r}P_{\delta_z}$, for each $n$, 
\begin{eqnarray*}
\hat{d}_n(t) &\leq& \max(V\beta+r_{\max}, \tilde{\alpha}_n^{d*}- \zeta/2)\\
 &\leq&\hat{d}_{\max}\triangleq\max(V\beta+r_{\max}, Vf_{\max}/\eta - \zeta/2). 
\end{eqnarray*}
The second inequality uses (\ref{eq:opt-multi-bdd}) in Appendix F. We now carry out a similar argument as in the proof of Theorem \ref{theorem:lram} and conclude that: 
\begin{eqnarray*}
\hspace{-.3in}\Delta_V(t)+\Delta_a(t) &\leq& G - 2N\delta_rd_{\max} - g(\hat{\bv{d}}(t), \hat{\bv{Q}}(t), \hat{\bv{H}}(t)) \\
&\leq& G - 2N\delta_r\hat{d}_{\max} - Vf_{\text{av}}^*. 
\end{eqnarray*}
 Carrying out a telescoping sum and taking a limit as in Theorem \ref{theorem:lram}'s proof, we obtain: 
\begin{eqnarray*}
\sum_nU_n(\overline{\gamma}_n) - \overline{c}\geq f_{\text{av}}^* - \frac{G+2N\hat{d}_{\max}\delta_r}{V}  - \lim_{T\rightarrow\infty}\frac{1}{T}\sum_{t=0}^{T-1}\expect{\Delta_a(t)}. 
\end{eqnarray*}
It remains to show that all the queues are finite. and that $\lim_{T\rightarrow\infty}\frac{1}{T}\sum_{t=0}^{T-1}\expect{\Delta_a(t)}=O(1/V)$. 
Then, we can conclude $\overline{r}_n\geq\overline{\gamma}_n$ and completes the proof. 

To this end, we first use (\ref{eq:pm_ineq}) and the definition of $\hat{\bv{d}}(t)$, $\hat{\bv{Q}}(t)$, and $\hat{\bv{H}}(t)$, to obtain that: 
\begin{eqnarray*}
\prob{ d_n(t)\geq \frac{3}{2}\zeta +D+\nu } &\leq& ae^{-K\nu} \\
\prob{ Q_n(t)\geq \frac{3}{2}\zeta +D+\nu } &\leq& ae^{-K\nu} \\
\prob{ H_m(t)\geq \frac{3}{2}\zeta +D+\nu } &\leq& ae^{-K\nu}, 
\end{eqnarray*}
which are exactly the queueing probability bounds (\ref{eq:dram-d-bdd}), (\ref{eq:dram-q-bdd}), and (\ref{eq:dram-h-bdd}).  
Using (\ref{eq:pm_ineq}) again, we see that for a large $V$ such that $\zeta\geq D+ \mu_{\max} +Nb_{\max}+r_{\max}+2\log(V)/K$, 
\begin{eqnarray}
\prob{ d_n(t)\leq r_{\max} } &\leq& \frac{a}{V^{2}} \label{eq:prob-bdd-a}\\
\prob{ Q_n(t)\leq  \mu_{\max} } &\leq& \frac{a}{V^{2}} \label{eq:prob-bdd-b}\\
\prob{ H_m(t)\leq Nb_{\max} } &\leq& \frac{a}{V^{2}}. \label{eq:prob-bdd-c} 
\end{eqnarray}
Combining the above bounds and Lemma \ref{lemma:q-avgrate}, and that $\hat{\bv{\alpha}}^*=O(V)$ by (\ref{eq:opt-multi-bdd}), we conclude that $\lim_{T\rightarrow\infty}\frac{1}{T}\sum_{t=0}^{T-1}\expect{\Delta_a(t)}=O(1/V)$. 
Moreover, since $\bv{d}(t)$ is stable, $\overline{r}_n\geq\overline{\gamma}_n$. Finally, using (\ref{eq:prob-bdd-a}) - (\ref{eq:prob-bdd-c}), we see that the fraction of time dropping happens, i.e., when the claimed reward does not  count, is $O(1/V^2)$. Since $\beta=\Theta(1)$, this  results in an additional utility loss of $O(1/V^2)$. 
Hence, we conclude that: 
\begin{eqnarray}
\sum_nU_n(\overline{r}_n) - \overline{c}\geq f_{\text{av}}^* - \frac{G+2N\hat{d}_{\max}\delta_r}{V} - O(1/V).  
\end{eqnarray}
This completes the proof. 
\end{proof}
%
%%Note that due to the use of $\hat{\bv{Q}}(t)$, $\hat{\bv{H}}(t)$, and $\hat{\bv{d}}(t)$, $\mathtt{DRAM}$'s 
\section*{Appendix E -- Proof of Theorem \ref{theorem:convergence-time}}
We prove Theorem \ref{theorem:convergence-time} here. 
\begin{proof} (Theorem \ref{theorem:convergence-time}) 
To prove the result, we define a different Lyapunov function as follows: 
\begin{eqnarray}
L_0(t) = \frac{1}{2}||(\bv{d}(t), \bv{Q}(t), \bv{H}(t))-\bv{\theta}- \tilde{\bv{\alpha}}^*||^2. 
\end{eqnarray}
Then, we define a one-slot conditional Lyapunov drift as $\Delta_0(t)=\expect{ L_0(t+1) - L_0(t) \left.|\right. \bv{y}(t) }$. Using the queueing dynamics, we obtain that: 
\begin{eqnarray}
\hspace{-.3in}&&\Delta_0(t) \leq G\\
\hspace{-.3in}&& - \sum_n (\tilde{\alpha}^{q*}_n - (Q_n(t)-\theta_1))\expect{\mu_n(t)-R_n(t)\left.|\right. \bv{y}(t)} \nonumber \\
\hspace{-.3in}&& - \sum_{n} (\tilde{\alpha}^{d*}_n - d_n(t))\expect{r_n(t)-\gamma_n(t)\left.|\right. \bv{y}(t)}   \nonumber\\
\hspace{-.3in}&& -\sum_m(\tilde{\alpha}^{h*}_m-(H_m(t)-\theta_2))\expect{\sum_{n}b_{mn}(t)-h_m(t)\left.|\right. \bv{y}(t)}.  \nonumber 
\end{eqnarray}
Using the fact that the last three components constitute the subgradient of $\tilde{g}(\bv{\alpha})$ at $\bv{\alpha} = (\bv{d}(t), \bv{Q}(t), \bv{H}(t))$ \cite{bertsekasoptbook}, we obtain: 
\begin{eqnarray*}
\hspace{-.3in}\Delta_0(t) &\leq& G - ( \tilde{g}(   (\bv{d}(t), \bv{Q}(t), \bv{H}(t)) -\bv{\theta} ) - \tilde{g}( \tilde{\bv{\alpha}}^*))  \nonumber \\
\hspace{-.3in} &\leq& G - \rho ||(\bv{d}(t), \bv{Q}(t), \bv{H}(t))-\bv{\theta}- \tilde{\bv{\alpha}}^*||. 
\end{eqnarray*}
Therefore, for any $0<\epsilon_0<\rho$, if $||(\bv{d}(t), \bv{Q}(t), \bv{H}(t))-\bv{\theta}- \tilde{\bv{\alpha}}^*||\geq \frac{G}{\rho-\epsilon_0}$, then the above implies that: 
\begin{eqnarray*}
\hspace{-.3in}&&\expect{||(\bv{d}(t+1), \bv{Q}(t+1), \bv{H}(t+1))-\bv{\theta}- \tilde{\bv{\alpha}}^*||^2 \left.|\right. \bv{y}(t)}\\
\hspace{-.3in}&\leq& ( ||(\bv{d}(t), \bv{Q}(t), \bv{H}(t))-\bv{\theta}- \tilde{\bv{\alpha}}^*|| - \epsilon_0)^2, 
\end{eqnarray*}
which further implies that: 
\begin{eqnarray*}
\hspace{-.3in}&&\expect{||(\bv{d}(t+1), \bv{Q}(t+1), \bv{H}(t+1))-\bv{\theta}- \tilde{\bv{\alpha}}^*|| \left.|\right. \bv{y}(t)}\\
\hspace{-.3in}&\leq& ||(\bv{d}(t), \bv{Q}(t), \bv{H}(t))-\bv{\theta}- \tilde{\bv{\alpha}}^*|| - \epsilon_0.  
\end{eqnarray*}
Then, using the fact that $\tilde{\bv{\alpha}}^*=\Theta(V)$ \cite{huangneely_dr_tac},  $(\bv{d}(T_{\delta_r}+1), \bv{Q}(T_{\delta_r}+1), \bv{H}(T_{\delta_r}+1) )=0$, and using Lemma $5$ in \cite{huang-learning-sig-14}, we conclude then: 
\begin{eqnarray*}
\expect{\tilde{T}^{\mathtt{LRAM}}_{D'_1}} = \expect{T_{\delta_r} +\Theta(V)}. 
\end{eqnarray*}
Here $D'_1=G/(\rho-\epsilon_0)=\Theta(1)$ and $\expect{\tilde{T}^{\mathtt{LRAM}}_{D'_1}}$  denotes the expected time to get to within $D'_1$ of $\tilde{\bv{\alpha}}^*$. Using (\ref{eq:dist-bdd2}) in Lemma \ref{lemma:dist-bdd}, and by defining $D_1 \triangleq D'_1+\frac{2Vf_{\max}\delta_r}{\rho\eta}$, 
we conclude that:  
\begin{eqnarray*}
\expect{T^{\mathtt{LRAM}}_{D_1}} = \expect{T_{\delta_r} +\Theta(V)}. \label{eq:conv-time-foo}
\end{eqnarray*}
This proves (\ref{eq:conv-lram}). To prove (\ref{eq:conv-dram}), note that the main difference between $\mathtt{DRAM}$ and $\mathtt{LRAM}$ is that  $\mathtt{DRAM}$ utilizes the system state information to ``jump start'' the algorithm. Using  Lemma \ref{lemma:dist-bdd} again, we see that with probability $P_{\delta_z}P_{\delta_r}$, $||(\bv{d}(0), \bv{Q}(0), \bv{H}(0))-\bv{\theta}- \tilde{\bv{\alpha}}^*||\leq \frac{2\delta_z  Vf_{\max}\vartheta}{\rho}$. Combing this result and (\ref{eq:conv-time-foo}), we conclude that: 
\begin{eqnarray*}
\expect{T^{\mathtt{DRAM}}_{D_2}} \leq \expect{T_{l} +\Theta(\frac{2\delta_z  Vf_{\max}\vartheta}{\rho\epsilon_0} ) }. \label{eq:conv-time-foo2}
\end{eqnarray*}
This proves (\ref{eq:conv-dram}) and completes the proof of the theorem. 
\end{proof}

\section*{Appendix F -- Proof of Lemma \ref{lemma:dist-bdd}} 
We present the proof for Lemma \ref{lemma:dist-bdd} here. For notation simplicity, we define $f_{\max}\triangleq N\beta r_{\max}+c_{\max}$. 
\begin{proof} (Lemma \ref{lemma:dist-bdd}) 
Since with probability $P_{\delta_r}P_{\delta_z}$, $\delta_r\leq\epsilon_r$ and $\delta_z\leq\epsilon_z$, we have from Assumption \ref{assumption:bdd-LM} that there exists a set of actions and probabilities that guarantee (\ref{eq:rate-service-a}), (\ref{eq:rate-a}), and (\ref{eq:resource-a}). Also, since $0<\sum_k\hat{\pi}_k\sum_i\lambda^k_i R^{k}_{in}<\expect{A_n(t)}$ and $0< \sum_k\hat{\pi}_k\sum_i\lambda^k_i h^{k}_{im} <\expect{e_m(t)}$, it can be shown that there exists $\eta_1=\Theta(1)>0$, such that for any subset $\script{I}_n\subset\script{N}$ and any subset $\script{I}_m\subset\script{M}$, there exist a set of actions $\{\hat{\bv{R}}^{k}_i\}_{k=1,..., |\script{Z}|}^{i=1,2, ..., \infty}$ and $\{\bv{h}^{k}_i\}_{k=1,..., |\script{Z}|}^{i=1,2, ..., \infty}$ such that: 
\begin{eqnarray}
  \sum_k\hat{\pi}_k\sum_i\lambda^k_i \hat{R}^{k}_{in} =  \sum_k\hat{\pi}_k\sum_i\lambda^k_i\mu_n(\bv{z}_k, \bv{b}^{k}_i)-\sigma_{\script{I}_n}\eta_1, \label{eq:rate-a2}
\end{eqnarray}
where $\sigma_{\script{I}_n}=1$ if $n\in\script{I}_n$ and $\sigma_{\script{I}_n}=-1$ otherwise. Similarly, 
\begin{eqnarray}
\hspace{-.3in}&& \sum_k\hat{\pi}_k\sum_i\lambda^k_i \sum_{n}b_{imn}^{k} =  \sum_k\hat{\pi}_k\sum_i\lambda^k_i h^{k}_{im} -\sigma_{\script{I}_m}\eta_1,\label{eq:resource-a2}
\end{eqnarray}
where $\sigma_{\script{I}_m}=1$ if $m\in\script{I}_m$ and $\sigma_{\script{I}_m}=-1$ otherwise.  
Then, using Lemma $1$ in \cite{huang-learning-sig-14}, we see that $\hat{\bv{\alpha}}^*$ obtained by solving (\ref{eq:dual-learning}) satisfies that: 
\begin{eqnarray}
\sum_n\hat{\alpha}_n^{d*}\eta + \sum_n|\hat{\alpha}_n^{q*}| \eta + \sum_m|\hat{\alpha}_m^{h*}|\eta\leq  Vf_{\max}. \label{eq:opt-multi-bdd}
\end{eqnarray}
Here $\eta=\min(\eta_0, \eta_1)$. Moreover, (\ref{eq:opt-multi-bdd}) also holds for $\bv{\alpha}^*$ and $\tilde{\bv{\alpha}}^*$. 
Now we look at $\tilde{g}(\tilde{\bv{\alpha}}^*)$ and $g(\bv{\alpha}^*)$. For explanation, we write $\tilde{g}(\tilde{\bv{\alpha}}^*) = \tilde{g}(\tilde{\bv{\alpha}}^*, \tilde{\bv{\gamma}}^*, \tilde{\bv{b}}^*, \tilde{\bv{R}}^*, \tilde{\bv{h}}^*)$, where $\tilde{\bv{\gamma}}^*, \tilde{\bv{b}}^*, \tilde{\bv{R}}^*, \tilde{\bv{h}}^*$ are the optimal actions corresponding to $\tilde{\bv{\alpha}}^*$ with $\hat{\bv{r}}$ and the true distribution $\bv{\pi}$. 
From the definition, we know that: 
\begin{eqnarray}
\hspace{-.2in}&&g(\bv{\alpha}^*, \bv{\gamma}^*, \bv{b}^*, \bv{R}^*, \bv{h}^*) \\
\hspace{-.2in}&\stackrel{(a)}{\geq}& \tilde{g}(\bv{\alpha}^*, \tilde{\bv{\gamma}}, \tilde{\bv{b}}, \tilde{\bv{R}}, \tilde{\bv{h}}) \nonumber\\
\hspace{-.2in}&& + \sum_k\pi_k \sum_{n}\alpha_n^{d*}[r_n(\bv{z}_k, \mu_n(z_k, \tilde{\bv{b}}^{k})) -\hat{r}_n(\bv{z}_k, \mu_n(z_k, \tilde{\bv{b}}^k))] \nonumber \\
\hspace{-.2in}&\stackrel{(b)}{\geq}&\tilde{g}(\tilde{\bv{\alpha}}^*, \tilde{\bv{\gamma}}^*, \tilde{\bv{b}}^*, \tilde{\bv{R}}^*, \tilde{\bv{h}}^*) - Vf_{\max}\delta_r/\eta \nonumber \\
\hspace{-.2in}&\stackrel{(c)}{\geq}&g(\tilde{\bv{\alpha}}^*, \hat{\bv{\gamma}}^*, \hat{\bv{b}}^*, \hat{\bv{R}}^*, \hat{\bv{h}}^*)  -  Vf_{\max}\delta_r /\eta \nonumber \\
\hspace{-.2in}&& + \sum_k\pi_k \sum_{n}\tilde{\alpha}_n^{d*}[\hat{r}_n(\bv{z}_k, \mu_n(z_k, \hat{\bv{b}}^{*k})) -r_n(\bv{z}_k, \mu_n(z_k, \hat{\bv{b}}^{*k}))] \nonumber \\
\hspace{-.2in}&\geq& g(\tilde{\bv{\alpha}}^*, \hat{\bv{\gamma}}^*, \hat{\bv{b}}^*, \hat{\bv{R}}^*, \hat{\bv{h}}^*) - 2Vf_{\max}\delta_r/\eta. \label{eq:g-bdd-0}
\end{eqnarray}
Here $\tilde{\bv{\gamma}}, \tilde{\bv{b}}, \tilde{\bv{R}}, \tilde{\bv{h}}$ denote the optimal actions corresponding to $\bv{\alpha}^*$ in $\tilde{g}(\bv{\alpha})$, and (a) follows from the definition of $\tilde{g}(\bv{\alpha})$ and the fact that $g(\bv{\alpha}^*)$ achieves the supremum over all actions. In (b), we have used the fact that $\tilde{g}(\tilde{\bv{\alpha}}^*)$ achieves the minimum over all $\bv{\alpha}$, that the learning module guarantees that $ \|\hat{\bv{r}} -  \bv{r} \|_{\max}\leq \delta_r$, and (\ref{eq:opt-multi-bdd}). Finally, in (c), $\hat{\bv{\gamma}}^*, \hat{\bv{b}}^*, \hat{\bv{R}}^*, \hat{\bv{h}}^*$ are the actions corresponding to $\tilde{\bv{\alpha}}^*$ under $g(\bv{\alpha})$ and it follows again because $\tilde{g}(\tilde{\bv{\alpha}}^*)$ achieves the supremum. The last inequality follows similarly to (b). 
Using the polyhedral structure of $g(\bv{\alpha})$, (\ref{eq:g-bdd-0}) implies that: 
\begin{eqnarray}
\|  \bv{\alpha}^*  - \tilde{\bv{\alpha}}^* \|\leq \frac{2Vf_{\max}\delta_r}{\rho\eta}. \label{eq:tilde-bdd}
\end{eqnarray}
This proves (\ref{eq:dist-bdd2}). 

To prove (\ref{eq:dist-bdd}), note that for any $\bv{\alpha}$, 
\begin{eqnarray}
\tilde{g}(\bv{\alpha}) - \hat{g}(\bv{\alpha}) = \sum_k(\pi_k - \hat{\pi}_k) \tilde{g}_k(\bv{\alpha}). 
\end{eqnarray}
Therefore, with probability $P_{\delta_z}$, we have: 
\begin{eqnarray*}
&&|\tilde{g}(\tilde{\bv{\alpha}}^*) - \hat{g}(\tilde{\bv{\alpha}}^*)| \\
&&\leq |\script{Z}|\delta_z\bigg(Vf_{\max} + \sum_n\tilde{\alpha}_n^{d*}r_{\max} +  \sum_n\tilde{\alpha}_n^{q*}(A_{\max}+\mu_{\max}) \\
&&+ \sum_m\tilde{\alpha}_m^{h*}(Nb_{\max}+h_{\max})  \bigg)\\
&&\leq   \delta_z  Vf_{\max}\vartheta, 
\end{eqnarray*}
where $\vartheta\triangleq |\script{Z}|(1+r_{\max} + A_{\max}+\mu_{\max} + Nb_{\max}+h_{\max})/\eta$ and the last inequality follows from (\ref{eq:opt-multi-bdd}). This then implies that: 
\begin{eqnarray}
|\tilde{g}(\tilde{\bv{\alpha}}^*) - \tilde{g}(\hat{\bv{\alpha}}^*)| \leq 2\delta_z  Vf_{\max}\vartheta, 
\end{eqnarray}
for otherwise we have: 
\begin{eqnarray*}
\hat{g}(\tilde{\bv{\alpha}}^*) -\hat{g}(\hat{\bv{\alpha}}^*)  \leq \tilde{g}(\tilde{\bv{\alpha}}^*) +\delta_z  Vf_{\max}\vartheta - \tilde{g}(\hat{\bv{\alpha}}^*)  + \delta_z  Vf_{\max}\vartheta <0, 
\end{eqnarray*}
which contradicts with the fact that $\hat{\bv{\alpha}}^*$ achieves the minimum of $\hat{g}(\bv{\alpha})$. Using the polyhedral structure, we see that  (\ref{eq:dist-bdd}) follows. 
This completes the proof of the lemma. 
\end{proof}

%\section*{Appendix ===}

\end{document}